\documentclass[a4paper,twoside,10pt]{amsart}
\usepackage{amsmath,amsfonts,amssymb,amsthm}
\newtheorem{theorem}{Theorem}[section]
\newtheorem{lemma}{Lemma}[section]
\usepackage{mathrsfs}
\usepackage{color,graphicx}
\usepackage[a4paper]{geometry}
\numberwithin{equation}{section}

\author[G. Nemes]{Gerg\H{o} Nemes}
\address{Central European University, Department of Mathematics and its Applications, H-1051 Budapest, N\'ador utca 9, Hungary}
\email{nemesgery@gmail.com}

\begin{document}

\title[Asymptotics of the gamma function and its reciprocal]{Error bounds and exponential improvements\\ for the asymptotic expansions of\\ the gamma function and its reciprocal}

\begin{abstract} In (Boyd, Proc. R. Soc. Lond. A \textbf{447} (1994) 609--630), W. G. C. Boyd derived a resurgence representation for the gamma function, exploiting the reformulation of the method of steepest descents by M. Berry and C. Howls (Berry and Howls, Proc. R. Soc. Lond. A \textbf{434} (1991) 657--675). Using this representation, he was able to derive a number of properties of the asymptotic expansion for the gamma function, including explicit and realistic error bounds, the smooth transition of the Stokes discontinuities, and asymptotics for the late coefficients. The main aim of this paper is to modify the resurgence formula of Boyd making it suitable for deriving better error estimates for the asymptotic expansions of the gamma function and its reciprocal. We also prove the exponentially improved versions of these expansions complete with error terms. Finally, we provide new (formal) asymptotic expansions for the coefficients appearing in the asymptotic series and compare their numerical efficacy with the results of earlier authors.
\end{abstract}
\maketitle

\section{Introduction}

It is well known that, as $z \to \infty$ in the sector $\left|\arg z\right| \leq \pi - \delta < \pi$, for any $0 < \delta \leq \pi$, the gamma function and its reciprocal have the following asymptotic expansions
\begin{equation}\label{eq1}
\Gamma \left( z \right) \sim \sqrt {2\pi} z^{z - \frac{1}{2}} e^{ - z} \sum\limits_{n = 0}^\infty  {\left( - 1 \right)^n \frac{\gamma _n}{z^n}} ,
\end{equation}
\begin{equation}\label{eq2}
\frac{1}{\Gamma \left( z \right)} \sim \frac{1}{\sqrt {2\pi}} z^{-z + \frac{1}{2}} e^z \sum\limits_{n = 0}^\infty  {\frac{\gamma _n}{z^n}} ,
\end{equation}
respectively. Here the $\gamma_n$'s are the so-called Stirling coefficients, the first few being $\gamma_0 = 1$ and
\[
\gamma _1  =  - \frac{1}{12},\quad \gamma _2  = \frac{1}{288},\quad \gamma _3  = \frac{139}{51840},\quad \gamma _4  =  - \frac{571}{2488320}.
\]
For a detailed discussion of the computation of these coefficients, see Appendix \ref{appendix}. The first proof of the expansion \eqref{eq1} for $z>0$ dates back to Laplace (see Copson \cite[p. 2]{Copson}). Since the 20th century, these expansions become standard textbook examples to illustrate various techniques, such as the method of Laplace itself or the method of steepest descents (see, for example, Copson \cite[pp. 53--58 and pp. 70--72]{Copson}, Paris \cite[pp. 24--28]{Paris} and Wong \cite[pp. 60--62 and pp. 110--111]{Wong}).

Error bounds for the expansion \eqref{eq1} were derived by Olver \cite{Olver1},\cite{Olver2} though, the application of these bounds requires the computation of extreme values of certain implicitly defined functions. It was not, however, until the end of the 20th century that simple, explicit error bounds for the asymptotic series \eqref{eq1} were found. Define for any $N\geq 1$ the remainder $R_N \left( z \right)$ by
\[
\Gamma \left( z \right) = \sqrt {2\pi} z^{z - \frac{1}{2}} e^{ - z} \left( {\sum\limits_{n = 0}^{N - 1} {\left(- 1 \right)^n \frac{\gamma _n }{z^n}}  + R_N \left( z \right)} \right) .
\]
Then Boyd \cite{Boyd2}, \cite[p. 141]{NIST} showed that
\begin{equation}\label{eq3}
\left| {R_N \left( z \right)} \right| \le \frac{\left( {1 + \zeta \left( N \right)} \right)\Gamma \left( N \right)}{\left( {2\pi } \right)^{N + 1} \left| z \right|^N}\frac{\min \left( {\sec \theta ,2\sqrt N } \right) + 1}{2} \; \text{ if } \; \left| \theta  \right| \le \frac{\pi }{2},
\end{equation}
where $\theta = \arg z$ and $\zeta$ denotes Riemann's zeta function. When $N=1$, the quantity $\zeta \left( N \right)$ has to be replaced by $3$.

Boyd's derivation of the error bound \eqref{eq3} is based on his resurgence formula for the gamma function coming from a general theory for complex Laplace-type integrals developed by Berry and Howls \cite{Berry} (see also Boyd \cite{Boyd1} and Paris \cite[pp. 94--99]{Paris}). Not just this error bound but the smooth transition of the Stokes discontinuities and the asymptotic behaviour of the coefficients $\gamma_n$ were discussed by Boyd using the resurgence formula.

The main goal of this paper is to modify the resurgence formula of Boyd making it suitable for deriving better error estimates for both \eqref{eq1} and \eqref{eq2} when $\Re\left(z\right)>0$. We also prove exponentially improved expansions for the gamma function and its reciprocal. Finally, we provide new (formal) asymptotic expansions for the Stirling coefficients and compare their numerical efficacy with the earlier results of Dingle and Boyd.

Similarly to $R_N\left(z\right)$, denote by $\widetilde R_N\left(z\right)$ the relative remainder of the series \eqref{eq2} after $N \geq 1$ terms, so that the last kept term is $\gamma_{N-1}z^{1-N}$. In the first theorem, we give bounds for the error terms $R_N\left(z\right)$ and $\widetilde R_N\left(z\right)$ when $z$ is real and positive.

\begin{theorem}\label{thm1} Suppose that $z>0$ and $N\geq 1$. Then
\begin{equation}\label{eq5}
\left( { - 1} \right)^N \gamma _{2N - 1}  \ge 0 \; \text{ and } \; \left( { - 1} \right)^{N + 1} \gamma _{2N}  \ge 0,
\end{equation}
and
\begin{equation}\label{eq6}
\left( { - 1} \right)^{N + 1} R_{2N - 1} \left( z \right) = \Theta _1 \left( {z,N} \right)\left( { - 1} \right)^N \frac{{\gamma _{2N - 1} }}{{z^{2N - 1} }} + \Theta _2 \left( {z,N} \right)\left( { - 1} \right)^{N + 1} \frac{{\gamma _{2N} }}{{z^{2N} }},
\end{equation}
\[
\left( { - 1} \right)^{N + 1} R_{2N} \left( z \right) = \Theta _2 \left( {z,N} \right)\left( { - 1} \right)^{N + 1} \frac{{\gamma _{2N} }}{{z^{2N} }} - \Theta_3 \left( {z,N} \right)\left( { - 1} \right)^{N + 1} \frac{{\gamma _{2N + 1} }}{{z^{2N + 1} }} ,
\]
\[
\left( { - 1} \right)^N \widetilde R_{2N - 1} \left( z \right) = \Theta _1 \left( {z,N} \right)\left( { - 1} \right)^N \frac{{\gamma _{2N - 1} }}{{z^{2N - 1} }} - \Theta _2 \left( {z,N} \right)\left( { - 1} \right)^{N + 1} \frac{{\gamma _{2N} }}{{z^{2N} }} ,
\]
\[
\left( { - 1} \right)^{N + 1} \widetilde R_{2N} \left( z \right) = \Theta _2 \left( {z,N} \right)\left( { - 1} \right)^{N + 1} \frac{{\gamma _{2N} }}{{z^{2N} }} + \Theta _3 \left( {z,N} \right)\left( { - 1} \right)^{N + 1} \frac{{\gamma _{2N + 1} }}{{z^{2N + 1} }} .
\]
Here $0 < \Theta_i \left( {z,N} \right) < 1$ ($i=1,2,3$) is a suitable number depending on $z$ and $N$. In particular, we have
\[
\left| R_{2N - 1} \left( z \right)\right| = \left( { - 1} \right)^{N + 1} R_{2N - 1} \left( z \right) < \left( { - 1} \right)^N \frac{\gamma _{2N - 1}}{z^{2N - 1}} + \left( { - 1} \right)^{N + 1} \frac{\gamma _{2N}}{z^{2N}},
\]
\[
\left|R_{2N} \left( z \right)\right| < \max\left(\left( { - 1} \right)^{N + 1} \frac{\gamma _{2N} }{z^{2N}},\left( { - 1} \right)^{N+1} \frac{\gamma _{2N + 1} }{z^{2N + 1}}\right),
\]
\[
\left| \widetilde R_{2N - 1} \left( z \right) \right| < \max \left( \left( { - 1} \right)^N \frac{{\gamma _{2N - 1} }}{{z^{2N - 1} }},\left( { - 1} \right)^{N + 1} \frac{\gamma _{2N}}{z^{2N}} \right),
\]
and
\[
\left| \widetilde R_{2N} \left( z \right) \right| = \left( { - 1} \right)^{N + 1} \widetilde R_{2N} \left( z \right) < \left( { - 1} \right)^{N + 1} \frac{\gamma _{2N}}{z^{2N}} + \left( { - 1} \right)^{N + 1} \frac{\gamma _{2N + 1}}{z^{2N + 1}} .
\]
\end{theorem}

In the second theorem, we provide bounds for the remainders $R_N\left(z\right)$ and $\widetilde R_N\left(z\right)$ assuming that $\Re\left(z\right)>0$.

\begin{theorem}\label{thm2} For any $N\geq 1$, we have
\begin{equation}\label{eq10}
\left| {R_N \left( z \right)} \right|,\left| {\widetilde R_N \left( z \right)} \right| \le \left( {\frac{{\left| {\gamma _N } \right|}}{{\left| z \right|^N }} + \frac{{\left| {\gamma _{N + 1} } \right|}}{{\left| z \right|^{N + 1} }}} \right) \begin{cases} \left|\csc\left(2\theta\right)\right| & \text{ if } \; \frac{\pi}{4} < \left|\theta\right| <\frac{\pi}{2} \\ 1 & \text{ if } \; \left|\theta\right| \leq \frac{\pi}{4}, \end{cases}
\end{equation}
where $\theta = \arg z$.
\end{theorem}

An asymptotic series for the logarithm of the gamma function which is analogous to \eqref{eq1} is given by
\begin{equation}\label{eq4}
\log \Gamma \left( z \right) \sim \left( {z - \frac{1}{2}} \right)\log z - z + \log \sqrt {2\pi } + \sum\limits_{n = 1}^\infty  {\frac{{B_{2n} }}{{2n\left( {2n - 1} \right)z^{2n - 1} }}} 
\end{equation}
as $z \to \infty$ in the sector $\left|\arg z\right| \leq \pi - \delta < \pi$, for any $0 < \delta \leq \pi$. Here $B_n$ stands for the $n$th Bernoulli number. Denoting by $r_N\left(z\right)$ the remainder after $N-1$ terms in this series, Lindel\"{o}f showed that
\[
\left| {r_N \left( z \right)} \right| \le \frac{\left|B_{2N}\right|}{{2N\left( {2N - 1} \right)\left| z \right|^{2N - 1} }} \begin{cases} \left|\csc\left(2\theta\right)\right| & \text{ if } \; \frac{\pi}{4} < \left|\theta\right| <\frac{\pi}{2} \\ 1 &  \text{ if } \; \left|\theta\right| \leq \frac{\pi}{4}, \end{cases}
\]
where $\theta = \arg z$ (see Remmert and Kay \cite[p. 67]{Remmert}). Also, if $z > 0$ is real then $r_N\left(z\right)$ is less than, but has the same sign as, the first neglected term (see, e.g., Temme \cite[p. 65]{Temme}). It is seen that our error bounds in Theorem \ref{thm1} and Theorem \ref{thm2} are the analogous of these results for the expansion \eqref{eq4}.

From the above remark on $r_N\left(z\right)$, it follows that for any $z>0$ we have $0 < r_1 \left( z \right) < \frac{B_2 }{2z} = \frac{1}{12z}$, whence
\[
1 < \frac{\Gamma \left( z \right)}{\sqrt {2\pi } z^{z - \frac{1}{2}} e^{ - z}} < e^{\frac{1}{12z}} = 1 + \frac{1}{12z} + \frac{1}{288z^2} + \frac{1}{10368z^3} +  \cdots.
\]
This is a well-known inequality (see, e.g., \cite[p. 138, equation 5.6.1]{NIST}). By Theorem \ref{thm1} we can improve the upper bound to
\[
1 < \frac{\Gamma \left( z \right)}{\sqrt {2\pi } z^{z - \frac{1}{2}} e^{ - z}} < 1 + \frac{1}{12z} + \frac{1}{288z^2}
\]
for any $z>0$; which is, as far as we know, a new identity. Thus, we have a simple estimate for the gamma function on the positive real line.

By the leading order behaviour of the Stirling coefficients (see \cite[p. 33]{Paris3}), we readily establish that the right-hand side of \eqref{eq10} is asymptotic to
\[
\left(\frac{{\left( {1 + \zeta \left( N \right)} \right)\Gamma \left( N \right)}}{{\left( {2\pi } \right)^{N + 1} \left| z \right|^N }} + \frac{\pi }{{6N}}\frac{{\left( {1 + \zeta \left( N+1 \right)} \right)\Gamma \left( {N + 1} \right)}}{{\left( {2\pi } \right)^{N + 2} \left| z \right|^{N + 1} }}\right) \begin{cases} \left|\csc\left(2\theta\right)\right| & \text{ if } \; \frac{\pi}{4} < \left|\theta\right| <\frac{\pi}{2} \\ 1 & \text{ if } \; \left|\theta\right| \leq \frac{\pi}{4}, \end{cases}
\]
for odd $N$, and
\[
 \left(\frac{\pi }{{6N}}\frac{{\left( {1 + \zeta \left( N \right)} \right)\Gamma \left( N \right)}}{{\left( {2\pi } \right)^{N + 1} \left| z \right|^N }} + \frac{{\left( {1 + \zeta \left( N+1 \right)} \right)\Gamma \left( {N + 1} \right)}}{{\left( {2\pi } \right)^{N + 2} \left| z \right|^{N + 1} }}\right)\begin{cases} \left|\csc\left(2\theta\right)\right| & \text{ if } \; \frac{\pi}{4} < \left|\theta\right| <\frac{\pi}{2} \\ 1 & \text{ if } \; \left|\theta\right| \leq \frac{\pi}{4}, \end{cases}
\]
for even $N$. Since $1 < \frac{1}{2}\left( {\sec \theta  + 1} \right)$ if $0<\left|\theta\right| \leq \frac{\pi}{4}$, and $\left| {\csc \left( {2\theta } \right)} \right| < \frac{1}{2}\left( {\sec \theta  + 1} \right)$ if $\frac{\pi}{4} < \left|\theta\right| <\frac{\pi}{2}$, we infer that our bounds \eqref{eq10} are better than the bound \eqref{eq3} of Boyd if $N$ and $z$ are large and $\arg z$ is not too close to the imaginary axis.

When $\left|\arg z\right|$ is close to $\frac{\pi}{2}$, the error bound \eqref{eq3} becomes 
\begin{equation}\label{eq19}
\left| {R_1 \left( z \right)} \right| \le \frac{3}{2\pi ^2 \left| z \right|} \; \text{ and } \;
\left| {R_N \left( z \right)} \right| \le \frac{{\left( {1 + \zeta \left( N \right)} \right)\Gamma \left( N \right)}}{{\left( {2\pi } \right)^{N + 1} \left| z \right|^N }}\frac{{2\sqrt N  + 1}}{2} \; \text{ for } \; N \geq 2.
\end{equation}
Boyd does not actually prove these bounds, just mentions that the proof is similar to the proof of his bound for the error term of the large argument asymptotics of the Bessel function $K_\nu \left(z\right)$ given in an earlier paper of his \cite{Boyd3}. Up to the first few steps we can indeed mimic the proof presented in \cite{Boyd3}, but at one point we need non-trivial estimates for the gamma function along certain rays of the complex plane. Nevertheless, we shall give a possible proof of \eqref{eq19} for $N \geq 2$. The case $N=1$ remains unproved, though it is uninteresting for practical applications.

\begin{theorem}\label{thm3} Suppose that $N\geq 2$. If $\left|\arg z\right| \le \frac{\pi}{2}$, then
\[
\left| {R_N \left( z \right)} \right|,\left| {\widetilde R_N \left( z \right)} \right| \le \frac{{\left( {1 + \zeta \left( N \right)} \right)\Gamma \left( N \right)}}{{\left( {2\pi } \right)^{N + 1} \left| z \right|^N }}\frac{{2\sqrt N  + 1}}{2}.
\]
\end{theorem}

Through this paper we will use frequently the concept of the scaled gamma function $\Gamma^{\ast} \left( z \right)$ which is defined by
\[
\Gamma^{\ast} \left( z \right) = \frac{{\Gamma \left( z \right)}}{{\sqrt {2\pi } z^{z - \frac{1}{2}} e^{ - z} }}
\]
for $\left|\arg z\right| < \pi$. The asymptotic series \eqref{eq1}, \eqref{eq2} and the error bounds can be extended to other sectors of the complex plane via the continuation formulas
\begin{equation}\label{eq28}
\Gamma^\ast \left( z \right) = \frac{1}{1 - e^{ \pm 2\pi iz}}\frac{1}{\Gamma^\ast \left( {ze^{ \mp \pi i} } \right)}
\; \text{ and } \; \Gamma^\ast  \left( z \right) =  - e^{ \pm 2\pi iz} \Gamma^\ast  \left( {ze^{ \pm 2\pi i} } \right) .
\end{equation}
Lines of the form $\arg z = \left( {2m \pm \frac{1}{2}} \right)\pi$, where $m\in \mathbb{Z}$, are the Stokes lines for the gamma function and its reciprocal.

In the next theorem we give exponentially improved asymptotic expansions for the gamma function and its reciprocal. The expansion for the gamma function can be viewed as the mathematically rigorous form of the terminated expansion of Dingle \cite[pp. 461--462]{Dingle}. We express these expansions in terms of the Terminant function $\widehat T_p\left(w\right)$ whose definition and basic properties are given in Section \ref{section3}. Throughout this paper, empty sums are taken to be zero.

\begin{theorem}\label{thm4} Suppose that $-\frac{3\pi}{2} \le \arg z \le \frac{3\pi}{2}$, $\left|z\right|$ is large and $N = 2\pi\left|z\right| + \rho$ is a positive integer with $\rho$ being bounded. Then
\begin{equation}\label{eq48}
R_N \left( z \right) = e^{2\pi iz} \sum\limits_{m = 0}^{M - 1} {\left( { - 1} \right)^m \frac{{\gamma _m }}{{z^m }}\widehat T_{N - m} \left( {2\pi iz} \right)}  - e^{ - 2\pi iz} \sum\limits_{m = 0}^{M - 1} {\left( { - 1} \right)^m \frac{{\gamma _m }}{{z^m }}\widehat T_{N - m} \left( { - 2\pi iz} \right)}  + R_{N,M} \left( z \right)
\end{equation}
and
\begin{equation}\label{eq30}
\widetilde R_N \left( z \right) =  - e^{2\pi iz} \sum\limits_{m = 0}^{M - 1} {\frac{{\gamma _m }}{{z^m }}\widehat T_{N - m} \left( {2\pi iz} \right)}  + e^{ - 2\pi iz} \sum\limits_{m = 0}^{M - 1} {\frac{{\gamma _m }}{{z^m }}\widehat T_{N - m} \left( { - 2\pi iz} \right)}  + \widetilde R_{N,M} \left( z \right),
\end{equation}
with $M\geq 0$ being an arbitrary fixed integer, and
\begin{equation}\label{eq29}
R_{N,M} \left( z \right),\widetilde R_{N,M} \left( z \right) = \mathcal{O}_{M,\rho } \left( {\frac{{e^{ - 2\pi \left| z \right|} }}{{\left| z \right|^M }}} \right)
\end{equation}
for $\left|\arg z\right| \le \frac{\pi}{2}$;
\[
R_{N,M} \left( z \right) = \mathcal{O}_{M,\rho } \left( {e^{ \mp 2\pi \Im \left( z \right)} \left( {\frac{1}{{\left| {1 - e^{ \mp 2\pi iz} } \right|}} + \frac{1}{{\left| z \right|^M }}} \right)} \right) \; \text{ and } \; \widetilde R_{N,M} \left( z \right) = \mathcal{O}_{M,\rho } \left( {\frac{{e^{ \mp 2\pi \Im \left( z \right)} }}{{\left| z \right|^M }}} \right)
\]
for $\frac{\pi}{2} \le \pm \arg z \le \frac{3\pi}{2}$.
\end{theorem}

The expansion \eqref{eq48} without the error term $R_{N,M} \left( z \right)$ was also derived by Boyd, but he mistakenly gave the sign of the factor $e^{ - 2\pi iz}$ to be positive.

For exponentially improved asymptotic expansions using Hadamard series, see Paris \cite{Paris4}, \cite[pp. 156--159]{Paris}.

While proving Theorem \ref{thm4} in Section \ref{section3}, we also obtain the following explicit bounds for the remainders in \eqref{eq48} and \eqref{eq30}. Note that in this theorem $N$ may not depend on $z$.

\begin{theorem}\label{thm5}
For any integers $2 \leq M < N$, define the remainders $R_{N,M} \left( z \right)$ and $\widetilde R_{N,M} \left( z \right)$ by the equations \eqref{eq48} and \eqref{eq30}, respectively. Then we have
\begin{multline}\label{eq54}
\left| {R_{N,M} \left( z \right)} \right|, \left| {\widetilde R_{N,M} \left( z \right)} \right| \le  \left( {6M + 2} \right)\frac{{\zeta \left( M \right)\Gamma \left( M \right)\Gamma \left( {N - M} \right)}}{{\left( {2\pi } \right)^{N + 2} \left| z \right|^N }}\\
+ \left( {2\sqrt M  + 1} \right)\frac{{\zeta \left( M \right)\Gamma \left( M \right)}}{{\left( {2\pi } \right)^{M + 1} \left| z \right|^M }}\left( {\left| {e^{2\pi iz} \widehat T_{N - M} \left( {2\pi iz} \right)} \right| + \left| {e^{ - 2\pi iz} \widehat T_{N - M} \left( { - 2\pi iz} \right)} \right|} \right),
\end{multline}
provided that $\left|\arg z\right| \leq \frac{\pi}{2}$.
\end{theorem}

The rest of the paper is organised as follows. In Section \ref{section2}, we prove the error bounds stated in Theorems \ref{thm1}--\ref{thm3}. In the first part of Section \ref{section3}, we prove the exponentially improved expansions given in Theorem \ref{thm4} and the error bounds given in Theorem \ref{thm5}. In the second part, we reveal some interesting facts about the Stokes phenomenon for the gamma function and its reciprocal. We also discuss the smooth transition of the Stokes discontinuities. In Section \ref{section4}, we derive new asymptotic approximations for the Stirling coefficients $\gamma_n$ and compare their numerical efficacy with the earlier results of Dingle and Boyd.

\section{Proofs of the error bounds}\label{section2}

Recall that for any $N\geq 1$ the remainder terms $R_N \left( z \right)$ and $\widetilde R_N \left( z \right)$ are defined by
\[
\Gamma^\ast \left( z \right) = \sum\limits_{n = 0}^{N - 1} {\left( { - 1} \right)^n \frac{{\gamma _n }}{{z^n }}}  + R_N \left( z \right) \; \text{ and } \; \frac{1}{{\Gamma^\ast \left( z \right)}} = \sum\limits_{n = 0}^{N - 1} {\frac{{\gamma _n }}{{z^n }}}  + \widetilde R_N \left( z \right) .
\]
Suppose that $\Re\left(z\right)>0$. Boyd's resurgence formulas \cite[equations (2.14) and (4.2)]{Boyd2} can be written in the forms
\begin{equation}\label{eq15}
R_N \left( z \right) = \frac{1}{{2\pi i}}\frac{{i^N }}{{z^N }}\int_0^{ + \infty } {\frac{{s^{N - 1} e^{ - 2\pi s} \Gamma^\ast  \left( {is} \right)}}{{1 - is/z}}ds}  - \frac{1}{{2\pi i}}\frac{{\left( { - i} \right)^N }}{{z^N }}\int_0^{ + \infty } {\frac{{s^{N - 1} e^{ - 2\pi s} \Gamma^\ast  \left( { - is} \right)}}{{1 + is/z}}ds} 
\end{equation}
and
\[
\widetilde R_N \left( z \right) = \frac{1}{{2\pi i}}\frac{{\left( { - i} \right)^N }}{{z^N }}\int_0^{ + \infty } {\frac{{s^{N - 1} e^{ - 2\pi s} \Gamma^\ast \left( {is} \right)}}{{1 + is/z}}ds}  - \frac{1}{{2\pi i}}\frac{{i^N }}{{z^N }}\int_0^{ + \infty } {\frac{{s^{N - 1} e^{ - 2\pi s} \Gamma^\ast  \left( { - is} \right)}}{{1 - is/z}}ds} .
\]
We remark that he stated the formula $\widetilde R_N \left( z \right)$ only for $N=1$, but the formula for the general $N$ follows easily from it. Using the property $\Gamma^{\ast} \left( {\bar z} \right) = \overline {\Gamma^{\ast} \left( z \right)}$, we deduce that
\begin{equation}\label{eq9}
\left( { - 1} \right)^{N + 1} R_{2N - 1} \left( z \right) = \frac{1}{{\pi z^{2N - 1} }}\int_0^{ + \infty } {\frac{{s^{2N - 2} e^{ - 2\pi s} \Re \Gamma ^{\ast}  \left( {is} \right)}}{{1 + \left( {s/z} \right)^2 }}ds}  - \frac{1}{{\pi z^{2N} }}\int_0^{ + \infty } {\frac{{s^{2N - 1} e^{ - 2\pi s} \Im \Gamma ^{\ast}  \left( {is} \right)}}{{1 + \left( {s/z} \right)^2 }}ds} ,
\end{equation}
\[
\left( { - 1} \right)^{N + 1} R_{2N} \left( z \right) = -\frac{1}{{\pi z^{2N} }}\int_0^{ + \infty } {\frac{{s^{2N - 1} e^{ - 2\pi s} \Im \Gamma ^{\ast}  \left( {is} \right)}}{{1 + \left( {s/z} \right)^2 }}ds} - \frac{1}{{\pi z^{2N + 1} }}\int_0^{ + \infty } {\frac{{s^{2N} e^{ - 2\pi s} \Re \Gamma ^{\ast}  \left( {is} \right)}}{{1 + \left( {s/z} \right)^2 }}ds},
\]
\[
\left( { - 1} \right)^N \widetilde R_{2N - 1} \left( z \right) = \frac{1}{{\pi z^{2N - 1} }}\int_0^{ + \infty } {\frac{{s^{2N - 2} e^{ - 2\pi s} \Re \Gamma ^{\ast}  \left( {is} \right)}}{{1 + \left( {s/z} \right)^2 }}ds}  + \frac{1}{{\pi z^{2N} }}\int_0^{ + \infty } {\frac{{s^{2N - 1} e^{ - 2\pi s} \Im \Gamma ^{\ast}  \left( {is} \right)}}{{1 + \left( {s/z} \right)^2 }}ds} ,
\]
\[
\left( { - 1} \right)^{N+1} \widetilde R_{2N} \left( z \right) = -\frac{1}{{\pi z^{2N} }}\int_0^{ + \infty } {\frac{{s^{2N - 1} e^{ - 2\pi s} \Im \Gamma ^{\ast}  \left( {is} \right)}}{{1 + \left( {s/z} \right)^2 }}ds} + \frac{1}{{\pi z^{2N + 1} }}\int_0^{ + \infty } {\frac{{s^{2N} e^{ - 2\pi s} \Re \Gamma ^{\ast}  \left( {is} \right)}}{{1 + \left( {s/z} \right)^2 }}ds},
\]
for any $N\geq 1$. These are the suitable forms of the remainders to obtain the realistic error bounds stated in Theorem \ref{thm1} and Theorem \ref{thm2}. From these and the formula $\gamma _N  = z^N \left( \widetilde R_N \left( z \right) - \widetilde R_{N + 1} \left( z \right) \right)$, we infer that
\begin{equation}\label{eq7}
\left( { - 1} \right)^N \gamma _{2N - 1}  = \frac{1}{\pi}\int_0^{ + \infty } {s^{2N - 2} e^{ - 2\pi s} \Re \Gamma ^{\ast}  \left( {is} \right)ds} 
\end{equation}
and
\begin{equation}\label{eq8}
\left( { - 1} \right)^{N+1} \gamma _{2N}  = -\frac{1}{\pi}\int_0^{ + \infty } {s^{2N - 1} e^{ - 2\pi s} \Im \Gamma ^{\ast}  \left( {is} \right)ds} ,
\end{equation}
for all $N\geq 1$. To complete the proof of Theorem \ref{thm1} and Theorem \ref{thm2}, we need the following lemma.

\begin{lemma}\label{lemma1} For any $s>0$ it holds that $\Re \Gamma^\ast  \left( {is} \right)\geq 0$ and $- \Im \Gamma^\ast  \left( {is} \right) \ge 0$.
\end{lemma}

\begin{proof} The proof is based on the following representation of $\Gamma^\ast  \left( z \right)$ due to Stieltjes:
\begin{equation}\label{eq13}
\Gamma^\ast  \left( z \right) = \exp \left( {\int_0^{ + \infty } {\frac{{Q\left( t \right)}}{{\left( {z + t} \right)^2 }}dt} } \right) \; \text{ for } \; \left|\arg z\right| < \pi,
\end{equation}
where $Q\left( t \right) = \frac{1}{2}\left( {t - \left\lfloor t \right\rfloor  - \left( {t - \left\lfloor t \right\rfloor } \right)^2 } \right)$ (see, e.g., Remmert and Kay \cite[pp. 56--58]{Remmert}). We shall use the fact that $0 \le Q\left( t \right) \le \frac{1}{8}$. Substituting $z=is$ with $s>0$ gives
\[
\Gamma^\ast \left( {is} \right) = \exp \left( { - \int_0^{ + \infty } {\frac{{s^2  - t^2 }}{{\left( {s^2  + t^2 } \right)^2 }}Q\left( t \right)dt}  - i\int_0^{ + \infty } {\frac{{2st}}{{\left( {s^2  + t^2 } \right)^2 }}Q\left( t \right)dt} } \right),
\]
whence
\[
\Re \Gamma^\ast  \left( {is} \right)  = \exp \left( { - \int_0^{ + \infty } {\frac{{s^2  - t^2 }}{{\left( {s^2  + t^2 } \right)^2 }}Q\left( t \right)dt} } \right)\cos \left( { \int_0^{ + \infty } {\frac{{2st}}{{\left( {s^2  + t^2 } \right)^2 }}Q\left( t \right)dt} } \right),
\]
and
\[
-\Im \Gamma^\ast  \left( {is} \right) = \exp \left( { - \int_0^{ + \infty } {\frac{{s^2  - t^2 }}{{\left( {s^2  + t^2 } \right)^2 }}Q\left( t \right)dt} } \right)\sin \left( { \int_0^{ + \infty } {\frac{{2st}}{{\left( {s^2  + t^2 } \right)^2 }}Q\left( t \right)dt} } \right).
\]
To prove the lemma, it is enough to show that the integral under the trigonometric functions is non-negative and is at most $\frac{\pi}{2}$ for any $s>0$. As $Q\left(t\right)$ is non-negative, the integral is non-negative. On the other hand,
\[
\int_0^1 {\frac{{2st}}{{\left( {s^2  + t^2 } \right)^2 }}Q\left( t \right)dt}  = \frac{s}{2}\log \left( {\frac{{s^2 }}{{s^2  + 1}}} \right) + \frac{1}{2}\arctan \left( {\frac{1}{s}} \right) \le \frac{\pi }{4},
\]
and
\[
\int_1^{ + \infty } {\frac{{2st}}{{\left( {s^2  + t^2 } \right)^2 }}Q\left( t \right)dt}  \le \int_1^{ + \infty } {\frac{{2st}}{{\left( {s^2  + t^2 } \right)^2 }}\frac{1}{8}dt}  = \frac{1}{8}\frac{s}{{s^2  + 1}} \le \frac{1}{{16}},
\]
whence
\[
\int_0^{ + \infty } {\frac{{2st}}{{\left( {s^2  + t^2 } \right)^2 }}Q\left( t \right)dt}  \le \frac{\pi }{4} + \frac{1}{{16}} < \frac{\pi }{2},
\]
for any $s>0$.
\end{proof}

The inequalities in \eqref{eq5} follow from the lemma and the representations \eqref{eq7} and \eqref{eq8}. From Theorem \ref{thm1} we prove only the bound \eqref{eq6}, the other results can be proved similarly. First, we note that
\[
0 < \frac{1}{{1 + \left( {s/z} \right)^2 }} < 1
\]
for any $s>0$ and $z>0$. Employing this inequality in \eqref{eq9} leads to
\[
\left( { - 1} \right)^{N + 1} R_{2N - 1} \left( z \right) = \frac{{\Theta _1 \left( {z,N} \right)}}{{\pi z^{2N - 1} }}\int_0^{ + \infty } {s^{2N - 2} e^{ - 2\pi s} \Re \Gamma^\ast  \left( {is} \right)ds}  - \frac{{\Theta _2 \left( {z,N} \right)}}{{\pi z^{2N} }}\int_0^{ + \infty } {s^{2N - 1} e^{ - 2\pi s} \Im \Gamma^\ast  \left( {is} \right)ds} ,
\]
where $\Theta _1 \left( {z,N} \right)$ and $\Theta _2 \left( {z,N} \right)$ are some functions of $z$ and $N$ satisfying $0<\Theta _1 \left( {z,N} \right)<1$ and $0<\Theta _2 \left( {z,N} \right)<1$. Upon inserting the formulas \eqref{eq7} and \eqref{eq8} into this representation we obtain \eqref{eq6}.

As for Theorem \ref{thm2}, we prove the result for $R_{2N - 1} \left( z \right)$, the proofs of the other bounds are similar. From \eqref{eq9} and Lemma \ref{lemma1} it follows that
\begin{equation}\label{eq11}
\left| {R_{2N - 1} \left( z \right)} \right| \le \frac{1}{{\pi \left| z \right|^{2N - 1} }}\int_0^{ + \infty } {\frac{{s^{2N - 2} e^{ - 2\pi s} \Re \Gamma^\ast  \left( {is} \right)}}{{\left| {1 + \left( {s/z} \right)^2 } \right|}}ds}  - \frac{1}{{\pi \left| z \right|^{2N} }}\int_0^{ + \infty } {\frac{{s^{2N - 1} e^{ - 2\pi s} \Im \Gamma^\ast \left( {is} \right)}}{{\left| {1 + \left( {s/z} \right)^2 } \right|}}ds} .
\end{equation}
It is easy to show that for any $r>0$
\[
\frac{1}{\left| {1 + re^{ - 2\theta i} } \right|} \le \begin{cases} \left|\csc\left(2\theta\right)\right| & \text{ if } \; \frac{\pi}{4} < \left|\theta\right| <\frac{\pi}{2} \\ 1 & \text{ if } \; \left|\theta\right| \leq \frac{\pi}{4}. \end{cases}
\]
Applying this inequality and the formulas \eqref{eq7} and \eqref{eq8} in \eqref{eq11}, proves the estimate \eqref{eq10} for the case of $R_{2N - 1} \left( z \right)$.

To prove Theorem \ref{thm3}, we shall use the lemma below.

\begin{lemma}\label{lemma2} For any $s>0$ and $0 <\varphi <\frac{\pi}{2}$, we have
\begin{equation}\label{eq14}
\left| {\Gamma^\ast \left( {\frac{{ise^{i\varphi } }}{{\cos \varphi }}} \right)} \right| \le \frac{1}{{\sqrt {1 - 2e^{ - 2\pi s} \cos \left( {2\pi s\tan \varphi } \right) + e^{ - 4\pi s} } }} \le \frac{1}{{1 - e^{ - 2\pi s} }} .
\end{equation}
\end{lemma}

\begin{proof}
An application of the reflection formula \eqref{eq28} for the gamma function and the relation $\overline {\Gamma^\ast \left( z \right)}  = \Gamma^\ast \left( {\bar z} \right)$ shows that
\begin{align*}
\log \left| {\Gamma^\ast \left( {\frac{{ise^{i\varphi } }}{{\cos \varphi }}} \right)} \right| & =  - \frac{1}{2}\log \left( {1 - 2e^{ - 2\pi s} \cos \left( {2\pi s\tan \varphi } \right) + e^{ - 4\pi s} } \right) - \log \left| {\Gamma^\ast \left( { - \frac{{ise^{i\varphi } }}{{\cos \varphi }}} \right)} \right| \\
& =  - \frac{1}{2}\log \left( {1 - 2e^{ - 2\pi s} \cos \left( {2\pi s\tan \varphi } \right) + e^{ - 4\pi s} } \right) - \log \left| {\Gamma^\ast \left( {\frac{{ise^{ - i\varphi } }}{{\cos \varphi }}} \right)} \right| .
\end{align*}
From this, we infer that
\begin{equation}\label{eq12}
\left| {\Gamma^\ast \left( {\frac{{ise^{i\varphi } }}{{\cos \varphi }}} \right)} \right| = \frac{1}{{\sqrt {1 - 2e^{ - 2\pi s} \cos \left( {2\pi s\tan \varphi } \right) + e^{ - 4\pi s} } }}\frac{1}{{\left| {\Gamma^\ast \left( {\frac{{ise^{ - i\varphi } }}{{\cos \varphi }}} \right)} \right|}} \le \frac{1}{{1 - e^{ - 2\pi s} }}\frac{1}{{\left| {\Gamma^\ast \left( {\frac{{ise^{ - i\varphi } }}{{\cos \varphi }}} \right)} \right|}} .
\end{equation}
Let $z=x+iy$ such that $\left|\arg z \right| < \frac{\pi}{2}$. We show that $\left| {\frac{1}{{\Gamma^\ast \left( z \right)}}} \right|$ is bounded in the right-half plane. Indeed, if $z$ is not too close to the origin, then by Stieltjes's formula \eqref{eq13}
\begin{align*}
\left| {\frac{1}{{\Gamma^\ast \left( z \right)}}} \right| & \le \exp \left( {\frac{1}{8}\int_0^{ + \infty } {\frac{{dt}}{{\left| {z + t} \right|^2 }}} } \right) \le \exp \left( {\frac{1}{{8\cos ^2 \left( {\frac{\theta }{2}} \right)}}\int_0^{ + \infty } {\frac{{dt}}{{\left( {\left| z \right| + t} \right)^2 }}} } \right) \\ & = \exp \left( {\frac{1}{{8\left| z \right|\cos ^2 \left( {\frac{\theta }{2}} \right)}}} \right) \le \exp \left( {\frac{1}{{4\left| z \right|}}} \right) \; \text{ with } \; \theta = \arg z.
\end{align*}
To see the boundedness near the origin, we note that
\begin{align*}
\left| {\frac{1}{{\Gamma^\ast \left( z \right)}}} \right| = \left| {\frac{{z^{z + \frac{1}{2}} \sqrt {2\pi } }}{{e^z }}} \right|\left| {\frac{1}{{\Gamma \left( {z + 1} \right)}}} \right|&  = \sqrt {2\pi } e^{ - \frac{\pi }{2}\left| y \right| + y\arctan \left( {\frac{x}{y}} \right) - x} \left| z \right|^{x + \frac{1}{2}} \left| {\frac{1}{{\Gamma \left( {z + 1} \right)}}} \right| \\
& \le \sqrt {2\pi } e^{ - x} \left| z \right|^{x + \frac{1}{2}} \left| {\frac{1}{{\Gamma \left( {z + 1} \right)}}} \right|,
\end{align*}
and that the reciprocal gamma function is an entire function. Since $\frac{1}{{\Gamma^\ast \left( z \right)}}$ is holomorphic when $\left|\arg z \right| < \frac{\pi}{2}$, continuous on its boundary and
\[
\left| {\frac{1}{\Gamma^\ast  \left( { \pm iy} \right)}} \right| = \sqrt {1 - e^{ - 2\pi y} }  \le 1,
\]
by the Phragm\'en--Lindel\"of principle (\cite[p. 177]{Titchmarsh})
\[
\left| {\frac{1}{\Gamma^\ast  \left( z \right)}} \right| \le 1
\]
holds for any $z$ in the sector $\left|\arg z \right| \leq \frac{\pi}{2}$. Employing this inequality with $z = \frac{{ise^{ - i\varphi } }}{{\cos \varphi }}$ in \eqref{eq12} gives \eqref{eq14}.
\end{proof}

We prove the claimed bound only for $R_N \left( z \right)$, the proof for $\widetilde R_N \left( z \right)$ is completely analogous. Since $R_N \left( {\bar z} \right) = \overline {R_N \left( z \right)}$, we can assume that $0 \le \theta = \arg z \le \frac{\pi}{2}$. The idea is to rotate the path of integration through an angle $0 < \varphi <\frac{\pi}{2}$ in the first integral in \eqref{eq15}, to find
\[
R_N \left( z \right) = \frac{1}{{2\pi i}}\frac{{i^N }}{{z^N }}\int_0^{ + \infty e^{i\varphi } } {\frac{{s^{N - 1} e^{ - 2\pi s} \Gamma^\ast  \left( {is} \right)}}{{1 - is/z}}ds}  - \frac{1}{{2\pi i}}\frac{{\left( { - i} \right)^N }}{{z^N }}\int_0^{ + \infty } {\frac{{s^{N - 1} e^{ - 2\pi s} \Gamma^\ast  \left( { - is} \right)}}{{1 + is/z}}ds} .
\]
By analytic continuation, this expression is certainly valid when $0 \le \theta \le \frac{\pi}{2}$. Substituting $s = \frac{te^{i\varphi }}{\cos \varphi}$ in the first integral and using the inequalities
\[
\left| {\frac{1}{{1 + is/z}}} \right| \le 1\; \text{  and  } \; \left| {\frac{1}{{1 - ite^{i\varphi } /\cos \varphi z}}} \right| \le \sec \left( {\theta  - \varphi } \right),
\]
we find
\begin{align*}
\left|R_N \left( z \right)\right| \leq \; & \frac{{\sec \left( {\theta  - \varphi } \right)}}{{\cos ^N \varphi }}\frac{1}{{\left| z \right|^N }}\frac{1}{{2\pi }}\int_0^{ + \infty } {t^{N - 1} e^{ - 2\pi t} \left| {\Gamma^\ast  \left( {\frac{{ite^{i\varphi } }}{{\cos \varphi }}} \right)} \right|dt} \\ & + \frac{1}{{\left| z \right|^N }}\frac{1}{{2\pi }}\int_0^{ + \infty } {s^{N - 1} e^{ - 2\pi s} \left| {\Gamma^\ast  \left( { - is} \right)} \right|ds} .
\end{align*}
The value $\varphi = \arctan\left(N^{-1/2}\right)$ minimises the function $\sec \left( {\frac{\pi }{2} - \varphi } \right)\cos ^{ - N} \varphi$, and
\[
\frac{{\sec \left( {\theta  - \arctan \left( {N^{ - 1/2} } \right)} \right)}}{{\cos ^N \left( {\arctan \left( {N^{ - 1/2} } \right)} \right)}} \le \frac{{\sec \left( {\frac{\pi }{2} - \arctan \left( {N^{ - 1/2} } \right)} \right)}}{{\cos ^N \left( {\arctan \left( {N^{ - 1/2} } \right)} \right)}} = \left( {1 + \frac{1}{N}} \right)^{\frac{{N + 1}}{2}} \sqrt N,
\]
for any $0 \le \theta \le \frac{\pi}{2}$ with $N\geq 1$. Boyd \cite[equation (3.9)]{Boyd2} showed that
\[
\frac{1}{{2\pi }}\int_0^{ + \infty } {s^{N - 1} e^{ - 2\pi s} \left| {\Gamma^\ast  \left( { - is} \right)} \right|ds}  \le \frac{{\left( {1 + \zeta \left( N \right)} \right)\Gamma \left( N \right)}}{{\left( {2\pi } \right)^{N + 1} }}\frac{1}{2} \; \text{ for } \; N\geq 2.
\]
From Lemma \ref{lemma2}, we obtain
\[
\frac{1}{{2\pi }}\int_0^{ + \infty } {t^{N - 1} e^{ - 2\pi t} \left| {\Gamma^\ast  \left( {\frac{{ite^{i\varphi } }}{{\cos \varphi }}} \right)} \right|dt}  \le \frac{1}{{2\pi }}\int_0^{ + \infty } {\frac{{t^{N - 1} e^{ - 2\pi t} }}{{1 - e^{ - 2\pi t} }}dt}  = \frac{{B_N \Gamma \left( N \right)}}{{\left( {2\pi } \right)^{N + 1} }} \frac{1}{2}
\]
for $N\geq 3$ with $B_N  = 2 \zeta \left( N \right)$, using formula 25.5.1 in \cite[p. 604]{NIST}. This estimate also holds when $N=2$, but for this case, we derive a sharper bound using the first inequality in Lemma \ref{lemma2}:
\begin{multline*}
\frac{1}{{2\pi }}\int_0^{ + \infty } {t^{2 - 1} e^{ - 2\pi t} \left| {\Gamma^\ast \left( {\frac{{ite^{i\varphi } }}{{\cos \varphi }}} \right)} \right|dt} \\ \le \frac{1}{{2\pi }}\int_0^{ + \infty } {\frac{{te^{ - 2\pi t} }}{{\sqrt {1 - 2e^{ - 2\pi t} \cos \left( {2\pi t\tan \left( {\arctan \left( {2^{ - 1/2} } \right)} \right)} \right) + e^{ - 4\pi t} } }}dt}  = \frac{{B_2 \Gamma \left( 2 \right)}}{{\left( {2\pi } \right)^{2 + 1} }} \frac{1}{2}
\end{multline*}
with $B_2  = 2.81944984\ldots <2.82$. Therefore
\[
\left| {R_N \left( z \right)} \right| \le \frac{1}{2}\left( {\frac{{B_N }}{{1 + \zeta \left( N \right)}}\left( {1 + \frac{1}{N}} \right)^{\frac{{N + 1}}{2}} \sqrt N  + 1} \right)\frac{{\left( {1 + \zeta \left( N \right)} \right)\Gamma \left( N \right)}}{{\left( {2\pi } \right)^{N + 1} \left| z \right|^N }} .
\]
To complete the proof, we note that
\[
{\frac{{B_N }}{{1 + \zeta \left( N \right)}}\left( {1 + \frac{1}{N}} \right)^{\frac{{N + 1}}{2}} }  < 2
\]
for any $N\geq 2$.

\section{Exponentially improved asymptotic expansions}\label{section3}

We shall find it convenient to express our exponentially improved expansions in terms of the (scaled) Terminant function, which is defined by
\[
\widehat T_p \left( w \right) = \frac{{e^{\pi ip} w^{1 - p} e^{ - w} }}{{2\pi i}}\int_0^{ + \infty } {\frac{{t^{p - 1} e^{ - t} }}{w + t}dt} \; \text{ for } \; p>0 \; \text{ and } \; \left| \arg w \right| < \pi ,
\]
and by analytic continuation elsewhere. Olver \cite[equations (4.5) and (4.6)]{Olver4} showed that when $p \sim \left|w\right|$ and $w \to \infty$, we have
\begin{equation}\label{eq25}
ie^{ - \pi ip} \widehat T_p \left( w \right) = \begin{cases} \mathcal{O}\left( {e^{ - w - \left| w \right|} } \right) & \; \text{ if } \; \left| {\arg w} \right| \le \pi \\ \mathcal{O}\left(1\right) & \; \text{ if } \; - 3\pi  < \arg w \le  - \pi. \end{cases}
\end{equation}
Concerning the smooth transition of the Stokes discontinuities, we will use the more precise asymptotics
\begin{equation}\label{eq35}
\widehat T_p \left( w \right) = \frac{1}{2} + \frac{1}{2}\mathop{\text{erf}} \left( {c\left( \varphi  \right)\sqrt {\frac{1}{2}\left| w \right|} } \right) + \mathcal{O}\left( {\frac{{e^{ - \frac{1}{2}\left| w \right|c^2 \left( \varphi  \right)} }}{{\left| w \right|^{\frac{1}{2}} }}} \right)
\end{equation}
for $-\pi +\delta \leq \arg w \leq 3 \pi -\delta$, $0 < \delta  \le 2\pi$; and
\begin{equation}\label{eq26}
e^{ - 2\pi ip} \widehat T_p \left( w \right) =  - \frac{1}{2} + \frac{1}{2}\mathop{\text{erf}} \left( { - \overline {c\left( { - \varphi } \right)} \sqrt {\frac{1}{2}\left| w \right|} } \right) + \mathcal{O}\left( {\frac{{e^{ - \frac{1}{2}\left| w \right|\overline {c^2 \left( { - \varphi } \right)} } }}{{\left| w \right|^{\frac{1}{2}} }}} \right)
\end{equation}
for $- 3\pi  + \delta  \le \arg w \le \pi  - \delta$, $0 < \delta \le 2\pi$. Here $\varphi = \arg w$ and $\mathop{\text{erf}}$ denotes the Error function. The quantity $c\left( \varphi  \right)$ is defined implicitly by the equation
\[
\frac{1}{2}c^2 \left( \varphi  \right) = 1 + i\left( {\varphi  - \pi } \right) - e^{i\left( {\varphi  - \pi } \right)},
\]
and corresponds to the branch of $c\left( \varphi  \right)$ which has the following expansion in the neighbourhood of $\varphi = \pi$:
\begin{equation}\label{eq27}
c\left( \varphi  \right) = \left( {\varphi  - \pi } \right) + \frac{i}{6}\left( {\varphi  - \pi } \right)^2  - \frac{1}{{36}}\left( {\varphi  - \pi } \right)^3  - \frac{i}{{270}}\left( {\varphi  - \pi } \right)^4  +  \cdots .
\end{equation}
For complete asymptotic expansions, see Olver \cite{Olver5}. We remark that Olver uses the different notation $F_p \left( w \right) = ie^{ - \pi ip} \widehat T_p \left( w \right)$ for the Terminant function and the other branch of the function $c\left( \varphi  \right)$. For further properties of the Terminant function, see, for example, Paris and Kaminski \cite[Chapter 6]{Paris3}.

\subsection{Proof of the exponentially improved expansions} We start by proving the expansion \eqref{eq48} and the estimate \eqref{eq29} for the right-half plane. Let $0 \leq M < N$ be integers. First suppose, in addition, that $M \geq 2$. As in the proof of Theorem \ref{thm3}, we rotate the path of integration by an angle $0 < \varphi <\frac{\pi}{2}$ in the first integral of Boyd's resurgence formula \eqref{eq15} to find that for any $s>0$
\[
\Gamma^\ast \left( {is} \right) = \sum\limits_{m = 0}^{M - 1} {\left( { - 1} \right)^m \frac{{\gamma _m }}{{\left( {is} \right)^m }}}  + R_M \left( {is} \right),
\]
with
\begin{gather}\label{eq22}
\begin{split}
R_M \left( {is} \right) & = \frac{1}{{2\pi i}}\frac{1}{{s^M }}\int_0^{ + \infty e^{i\varphi } } {\frac{{t^{M - 1} e^{ - 2\pi t} \Gamma^\ast  \left( {it} \right)}}{{1 - t/s}}dt}  - \frac{1}{{2\pi i}}\frac{1}{{\left( { - s} \right)^M }}\int_0^{ + \infty } {\frac{{t^{M - 1} e^{ - 2\pi t} \Gamma^\ast \left( { - it} \right)}}{{1 + t/s}}dt} \\
& = \frac{1}{{2\pi i}}\frac{1}{{\left( {se^{ - i\varphi } } \right)^M }}\int_0^{ + \infty } {\frac{{t^{M - 1} e^{ - 2\pi te^{i\varphi } } \Gamma^\ast \left( {ite^{i\varphi } } \right)}}{{1 - te^{i\varphi } /s}}dt}  - \frac{1}{{2\pi i}}\frac{1}{{\left( { - s} \right)^M }}\int_0^{ + \infty } {\frac{{t^{M - 1} e^{ - 2\pi t} \Gamma^\ast \left( { - it} \right)}}{{1 + t/s}}dt} .
\end{split}
\end{gather}
A similar formula for $\Gamma^\ast \left( {-is} \right)$ can be obtained in the same way. First, we suppose that $\left|\arg z\right| < \frac{\pi}{2}$. Substitution into \eqref{eq15} yields
\begin{gather}\label{eq20}
\begin{split}
R_N \left( z \right) = \; & \sum\limits_{m = 0}^{M - 1} {\left( { - 1} \right)^m \frac{{\gamma _m }}{{z^m }}\frac{{i^{N - m} z^{m - N} }}{{2\pi i}}\int_0^{ + \infty } {\frac{{s^{N - m - 1} e^{ - 2\pi s} }}{{1 - is/z}}ds} } \\ & - \sum\limits_{m = 0}^{M - 1} {\left( { - 1} \right)^m \frac{{\gamma _m }}{{z^m }}\frac{{\left( { - i} \right)^{N - m} z^{m - N} }}{{2\pi i}}\int_0^{ + \infty } {\frac{{s^{N - m - 1} e^{ - 2\pi s} }}{{1 + is/z}}ds} }  + R_{N,M} \left( z \right)
\end{split}
\end{gather}
with
\begin{equation}\label{eq21}
R_{N,M} \left( z \right) = \frac{1}{{2\pi i}}\frac{{i^N }}{{z^N }}\int_0^{ + \infty } {\frac{{s^{N - 1} e^{ - 2\pi s} R_M \left( {is} \right)}}{{1 - is/z}}ds}  - \frac{1}{{2\pi i}}\frac{{\left( { - i} \right)^N }}{{z^N }}\int_0^{ + \infty } {\frac{{s^{N - 1} e^{ - 2\pi s} R_M \left( { - is} \right)}}{{1 + is/z}}ds} .
\end{equation}
The integrals in \eqref{eq20} can be identified in terms of the Terminant function since
\[
\frac{{\left( { \pm i} \right)^{N - m} z^{m - N} }}{{2\pi i}}\int_0^{ + \infty } {\frac{{s^{N - m - 1} e^{ - 2\pi s} }}{{1 \mp is/z}}ds}  = e^{ \pm 2\pi iz} \widehat T_{N - m} \left( { \pm 2\pi iz} \right).
\]
Therefore, we have the following expansion
\begin{equation}\label{eq49}
R_N \left( z \right) = e^{2\pi iz} \sum\limits_{m = 0}^{M - 1} {\left( { - 1} \right)^m \frac{{\gamma _m }}{{z^m }}\widehat T_{N - m} \left( {2\pi iz} \right)}  - e^{ - 2\pi iz} \sum\limits_{m = 0}^{M - 1} {\left( { - 1} \right)^m \frac{{\gamma _m }}{{z^m }}\widehat T_{N - m} \left( { - 2\pi iz} \right)}  + R_{N,M} \left( z \right).
\end{equation}
Taking $z = re^{i\theta }$, the representation \eqref{eq21} becomes
\begin{equation}\label{eq23}
R_{N,M} \left( z \right) = \frac{1}{{2\pi i}}\frac{1}{{\left( { - ie^{i\theta } } \right)^N }}\int_0^{ + \infty } {\frac{{\tau ^{N - 1} e^{ - 2\pi r\tau } R_M \left( {ir\tau } \right)}}{{1 - i\tau e^{ - i\theta } }}d\tau }  - \frac{1}{{2\pi i}}\frac{1}{{\left( {ie^{i\theta } } \right)^N }}\int_0^{ + \infty } {\frac{{\tau ^{N - 1} e^{ - 2\pi r\tau } R_M \left( { - ir\tau } \right)}}{{1 + i\tau e^{ - i\theta } }}d\tau } .
\end{equation}
We consider the first integral. Using the integral formula \eqref{eq22}, $R_M \left( {ir\tau } \right)$ can be written in the form
\begin{align*}
R_M \left( {ir\tau } \right) = \; & \frac{1}{{2\pi i}}\frac{1}{{\left( {r\tau e^{ - i\varphi } } \right)^M }}\int_0^{ + \infty } {\frac{{t^{M - 1} e^{ - 2\pi te^{i\varphi } } \Gamma^\ast \left( {ite^{i\varphi } } \right)}}{{1 - te^{i\varphi } /r\tau }}dt}  - \frac{1}{{2\pi i}}\frac{1}{{\left( { - r\tau } \right)^M }}\int_0^{ + \infty } {\frac{{t^{M - 1} e^{ - 2\pi t} \Gamma^\ast \left( { - it} \right)}}{{1 + t/r\tau }}dt} \\
 = \; & \frac{1}{{2\pi i}}\frac{1}{{\left( {r\tau e^{ - i\varphi } } \right)^M }}\left( {\int_0^{ + \infty } {\frac{{t^{M - 1} e^{ - 2\pi te^{i\varphi } } \Gamma^\ast \left( {ite^{i\varphi } } \right)}}{{1 - te^{i\varphi } /r}}dt}  + \left( {\tau  - 1} \right)\int_0^{ + \infty } {\frac{{t^{M - 1} e^{ - 2\pi te^{i\varphi } } \Gamma^\ast \left( {ite^{i\varphi } } \right)}}{{\left( {1 - r\tau /te^{i\varphi } } \right)\left( {1 - te^{i\varphi } /r} \right)}}dt} } \right) \\
& - \frac{1}{{2\pi i}}\frac{1}{{\left( { - r\tau } \right)^M }}\left( {\int_0^{ + \infty } {\frac{{t^{M - 1} e^{ - 2\pi t} \Gamma^\ast \left( { - it} \right)}}{{1 + t/r}}dt}  + \left( {\tau  - 1} \right)\int_0^{ + \infty } {\frac{{t^{M - 1} e^{ - 2\pi t} \Gamma^\ast \left( { - it} \right)}}{{\left( {1 + r\tau /t} \right)\left( {1 + t/r} \right)}}dt} } \right) .
\end{align*}
Therefore, the first integral in \eqref{eq23} can be estimated as follows
\begin{align*}
& \left| {\frac{1}{{2\pi i}}\frac{1}{{\left( { - ie^{i\theta } } \right)^N }}\int_0^{ + \infty } {\frac{{\tau ^{N - 1} e^{ - 2\pi r\tau } R_M \left( {ir\tau } \right)}}{{1 - i\tau e^{ - i\theta } }}d\tau } } \right| \\  & \le \frac{1}{{2\pi r^M }}\left| {\int_0^{ + \infty } {\frac{{t^{M - 1} e^{ - 2\pi te^{i\varphi } } \Gamma^\ast \left( {ite^{i\varphi } } \right)}}{{1 - te^{i\varphi } /r}}dt} } \right|\left| {\frac{1}{{2\pi }}\int_0^{ + \infty } {\frac{{\tau ^{N - M - 1} e^{ - 2\pi r\tau } }}{{1 - i\tau e^{ - i\theta } }}d\tau } } \right|\\
& + \frac{1}{{2\pi r^M }}\int_0^{ + \infty } {\tau ^{N - M - 1} e^{ - 2\pi r\tau } \left| {\frac{{\tau  - 1}}{{\tau  + ie^{i\theta } }}} \right|\left| {\frac{1}{{2\pi }}\int_0^{ + \infty } {\frac{{t^{M - 1} e^{ - 2\pi te^{i\varphi } } \Gamma^\ast \left( {ite^{i\varphi } } \right)}}{{\left( {1 - r\tau /te^{i\varphi } } \right)\left( {1 - te^{i\varphi } /r} \right)}}dt} } \right|d\tau } \\
& + \frac{1}{{2\pi r^M }}\left| {\int_0^{ + \infty } {\frac{{t^{M - 1} e^{ - 2\pi t} \Gamma^\ast \left( { - it} \right)}}{{1 + t/r}}dt} } \right|\left| {\frac{1}{{2\pi }}\int_0^{ + \infty } {\frac{{\tau ^{N - M - 1} e^{ - 2\pi r\tau } }}{{1 - i\tau e^{ - i\theta } }}d\tau } } \right|\\
& + \frac{1}{{2\pi r^M }}\int_0^{ + \infty } {\tau ^{N - M - 1} e^{ - 2\pi r\tau } \left| {\frac{{\tau  - 1}}{{\tau  + ie^{i\theta } }}} \right|\left| {\frac{1}{{2\pi }}\int_0^{ + \infty } {\frac{{t^{M - 1} e^{ - 2\pi t} \Gamma^\ast \left( { - it} \right)}}{{\left( {1 + r\tau /t} \right)\left( {1 + t/r} \right)}}dt} } \right|d\tau } .
\end{align*}
Noting that
\[
\left| {\frac{{\tau  - 1}}{{\tau  + ie^{i\theta } }}} \right| \le 1,\; \left| {\frac{1}{{1 + t/r}}} \right| \le 1,\; \left| {\frac{1}{{\left( {1 + r\tau /t} \right)\left( {1 + t/r} \right)}}} \right| \le 1
\]
and
\[
\left| {\frac{1}{{1 - te^{i\varphi } /r}}} \right| \le \csc \varphi ,\; \left| {\frac{1}{{\left( {1 - r\tau /te^{i\varphi } } \right)\left( {1 - te^{i\varphi } /r} \right)}}} \right| \le \csc ^2 \varphi 
\]
for any positive $r$, $t$ and $\tau$, we deduce the upper bound
\begin{align*}
& \left| {\frac{1}{{2\pi i}}\frac{1}{{\left( { - ie^{i\theta } } \right)^N }}\int_0^{ + \infty } {\frac{{\tau ^{N - 1} e^{ - 2\pi r\tau } R_M \left( {ir\tau } \right)}}{{1 - i\tau e^{ - i\theta } }}d\tau } } \right| \\ 
& \le \frac{{\csc \varphi }}{{2\pi r^M }}\int_0^{ + \infty } {t^{M - 1} \left| {e^{ - 2\pi te^{i\varphi } } \Gamma^\ast \left( {ite^{i\varphi } } \right)} \right|dt} \left| {\frac{1}{{2\pi }}\int_0^{ + \infty } {\frac{{\tau ^{N - M - 1} e^{ - 2\pi r\tau } }}{{1 - i\tau e^{ - i\theta } }}d\tau } } \right|\\
& + \frac{{\csc ^2 \varphi }}{{2\pi }}\int_0^{ + \infty } {t^{M - 1}\left| {e^{ - 2\pi te^{i\varphi } } \Gamma^\ast \left( {ite^{i\varphi } } \right)} \right|dt} \frac{1}{{2\pi r^M }}\int_0^{ + \infty } {\tau ^{N - M - 1} e^{ - 2\pi r\tau } d\tau } \\
& + \frac{1}{{2\pi r^M }}\int_0^{ + \infty } {t^{M - 1} e^{ - 2\pi t} \left| {\Gamma^\ast \left( { - it} \right)} \right|dt} \left| {\frac{1}{{2\pi }}\int_0^{ + \infty } {\frac{{\tau ^{N - M - 1} e^{ - 2\pi r\tau } }}{{1 - i\tau e^{ - i\theta } }}d\tau } } \right|\\
& + \frac{1}{{2\pi }}\int_0^{ + \infty } {t^{M - 1} e^{ - 2\pi t} \left| {\Gamma^\ast \left( { - it} \right)} \right|dt} \frac{1}{{2\pi r^M }}\int_0^{ + \infty } {\tau ^{N - M - 1} e^{ - 2\pi r\tau } d\tau }.
\end{align*}
A straightforward computation shows that this upper bound simplifies to
\begin{align*}
& \left| {\frac{1}{{2\pi i}}\frac{1}{{\left( { - ie^{i\theta } } \right)^N }}\int_0^{ + \infty } {\frac{{\tau ^{N - 1} e^{ - 2\pi r\tau } R_M \left( {ir\tau } \right)}}{{1 - i\tau e^{ - i\theta } }}d\tau } } \right| \\ 
& \le \left( {\frac{{\csc \varphi }}{{\cos ^M \varphi }}\frac{1}{{2\pi }}\int_0^{ + \infty } {t^{M - 1} e^{ - 2\pi t} \left| {\Gamma^\ast \left( {\frac{{ite^{i\varphi } }}{{\cos \varphi }}} \right)} \right|dt}  + \frac{1}{{2\pi }}\int_0^{ + \infty } {t^{M - 1} e^{ - 2\pi t} \left| {\Gamma^\ast \left( { - it} \right)} \right|dt} } \right)\frac{{\left| {e^{2\pi iz} \widehat T_{N - M} \left( {2\pi iz} \right)} \right|}}{{\left| z \right|^M }} \\
& + \left( {\frac{{\csc ^2 \varphi }}{{\cos ^M \varphi }}\frac{1}{{2\pi }}\int_0^{ + \infty } {t^{M - 1} e^{ - 2\pi t} \left| {\Gamma^\ast \left( {\frac{{ite^{i\varphi } }}{{\cos \varphi }}} \right)} \right|dt}  + \frac{1}{{2\pi }}\int_0^{ + \infty } {t^{M - 1} e^{ - 2\pi t} \left| {\Gamma^\ast \left( { - it} \right)} \right|dt} } \right)\frac{{\Gamma \left( {N - M} \right)}}{{\left( {2\pi } \right)^{N - M + 1} \left| z \right|^N }} .
\end{align*}
With the choice $\varphi = \arctan\left(M^{-\frac{1}{2}}\right)$, we have
\[
\frac{{\csc \varphi }}{{\cos ^M \varphi }} = \left( {\frac{{M + 1}}{M}} \right)^{\frac{{M + 1}}{2}} \sqrt M  < 2\sqrt M \; \text{ and } \; \frac{{\csc ^2 \varphi }}{{\cos ^M \varphi }} = \left( {\frac{{M + 1}}{M}} \right)^{\frac{M}{2} + 1} M < 3M .
\]
By Lemma \ref{lemma2}, we obtain the estimate
\[
\frac{1}{{2\pi }}\int_0^{ + \infty } {t^{M - 1} e^{ - 2\pi t} \left| {\Gamma^\ast \left( {\frac{{ite^{i\varphi } }}{{\cos \varphi }}} \right)} \right|dt}  \le \frac{1}{{2\pi }}\int_0^{ + \infty } {\frac{{t^{M - 1} e^{ - 2\pi t} }}{{1 - e^{ - 2\pi t} }}dt}  = \frac{{\zeta \left( M \right)\Gamma \left( M \right)}}{{\left( {2\pi } \right)^{M + 1} }} .
\]
The other type of integral can be bounded by the same quantity since
\[
\frac{1}{{2\pi }}\int_0^{ + \infty } {t^{M - 1} e^{ - 2\pi t} \left| {\Gamma^\ast  \left( { - it} \right)} \right|dt}  = \frac{1}{{2\pi }}\int_0^{ + \infty } {\frac{{t^{M - 1} e^{ - 2\pi t} }}{{\sqrt {1 - e^{ - 2\pi t} } }}dt}  < \frac{{\zeta \left( M \right)\Gamma \left( M \right)}}{{\left( {2\pi } \right)^{M + 1} }}.
\]
Therefore, we find
\begin{multline*}
 \left| {\frac{1}{{2\pi i}}\frac{1}{{\left( { - ie^{i\theta } } \right)^N }}\int_0^{ + \infty } {\frac{{\tau ^{N - 1} e^{ - 2\pi r\tau } R_M \left( {ir\tau } \right)}}{{1 - i\tau e^{ - i\theta } }}d\tau } } \right| \\
 \le \left( {2\sqrt M  + 1} \right)\frac{{\zeta \left( M \right)\Gamma \left( M \right)}}{{\left( {2\pi } \right)^{M + 1} \left| z \right|^M }}\left| {e^{2\pi iz} \widehat T_{N - M} \left( {2\pi iz} \right)} \right| + \left( {3M + 1} \right)\frac{{\zeta \left( M \right)\Gamma \left( M \right)\Gamma \left( {N - M} \right)}}{{\left( {2\pi } \right)^{N + 2} \left| z \right|^N }} .
\end{multline*}
Similarly, we have the following upper bound for the second integral in \eqref{eq23}:
\begin{multline*}
 \left| {\frac{1}{{2\pi i}}\frac{1}{{\left( { ie^{i\theta } } \right)^N }}\int_0^{ + \infty } {\frac{{\tau ^{N - 1} e^{ - 2\pi r\tau } R_M \left( {-ir\tau } \right)}}{{1 + i\tau e^{ - i\theta } }}d\tau } } \right| \\
 \le \left( {2\sqrt M  + 1} \right)\frac{{\zeta \left( M \right)\Gamma \left( M \right)}}{{\left( {2\pi } \right)^{M + 1} \left| z \right|^M }}\left| {e^{-2\pi iz} \widehat T_{N - M} \left( {-2\pi iz} \right)} \right| + \left( {3M + 1} \right)\frac{{\zeta \left( M \right)\Gamma \left( M \right)\Gamma \left( {N - M} \right)}}{{\left( {2\pi } \right)^{N + 2} \left| z \right|^N }}.
\end{multline*}
Thus, we conclude that
\begin{align*}
\left| {R_{N,M} \left( z \right)} \right| \le \; & \left( {2\sqrt M  + 1} \right)\frac{{\zeta \left( M \right)\Gamma \left( M \right)}}{{\left( {2\pi } \right)^{M + 1} \left| z \right|^M }}\left( {\left| {e^{2\pi iz} \widehat T_{N - M} \left( {2\pi iz} \right)} \right| + \left| {e^{ - 2\pi iz} \widehat T_{N - M} \left( { - 2\pi iz} \right)} \right|} \right)\\
& + \left( {6M + 2} \right)\frac{{\zeta \left( M \right)\Gamma \left( M \right)\Gamma \left( {N - M} \right)}}{{\left( {2\pi } \right)^{N + 2} \left| z \right|^N }} .
\end{align*}
By continuity, this bound holds in the closed sector $\left|\arg z\right| \leq \frac{\pi}{2}$. Suppose that $N = 2\pi \left| z \right| + \rho$ where $\rho$ is a bounded quantity. Employing Stirling's formula, we find that
\[
\left( {6M + 2} \right)\frac{{\zeta \left( M \right)\Gamma \left( M \right)\Gamma \left( {N - M} \right)}}{{\left( {2\pi } \right)^{N + 2} \left| z \right|^N }} = \mathcal{O}_{M,\rho } \left( {\frac{{e^{ - 2\pi \left| z \right|} }}{{\left| z \right|^{M + \frac{1}{2}} }}} \right)
\]
as $z \to \infty$. Olver's estimation \eqref{eq25} shows that
\[
\left( {2\sqrt M  + 1} \right)\frac{{\zeta \left( M \right)\Gamma \left( M \right)}}{{\left( {2\pi } \right)^{M + 1} \left| z \right|^M }}\left( {\left| {e^{2\pi iz} \widehat T_{N - M} \left( {2\pi iz} \right)} \right| + \left| {e^{ - 2\pi iz} \widehat T_{N - M} \left( { - 2\pi iz} \right)} \right|} \right) = \mathcal{O}_{M,\rho } \left( {\frac{{e^{ - 2\pi \left| z \right|} }}{{\left| z \right|^M }}} \right)
\]
for large $z$. Therefore, we have that
\begin{equation}\label{eq51}
R_{N,M} \left( z \right)= \mathcal{O}_{M,\rho } \left( {\frac{{e^{ - 2\pi \left| z \right|} }}{{\left| z \right|^M }}} \right)
\end{equation}
as $z \to \infty$ in the sector $\left|\arg z\right| \leq \frac{\pi}{2}$. If $M=0$ or $1$, we define $R_{N,M} \left( z \right)$ by \eqref{eq49}, consequently
\begin{equation}\label{eq52}
R_{N,1} \left( z \right) = \frac{{e^{ - 2\pi iz} \widehat T_{N - 1} \left( { - 2\pi iz} \right) - e^{2\pi iz} \widehat T_{N - 1} \left( {2\pi iz} \right)}}{{12z}} + R_{N,2} \left( z \right),
\end{equation}
\begin{equation}\label{eq53}
R_{N,0} \left( z \right) = e^{2\pi iz} \widehat T_N \left( {2\pi iz} \right) - e^{ - 2\pi iz} \widehat T_N \left( { - 2\pi iz} \right) + R_{N,1} \left( z \right).
\end{equation}
The proof of \eqref{eq29} for the cases $M=0,1$ now follows from these representations, the bound \eqref{eq51} we have established and Olver's estimate \eqref{eq25}.

The proof of the expansion \eqref{eq30} and the estimates \eqref{eq29}, \eqref{eq54} for the remainder $\widetilde R_{N,M} \left( z \right)$ is similar.

Consider now the sector $\frac{\pi}{2} < \arg z < \frac{3\pi}{2}$. Assume again that $M \geq 2$. When $z$ enters the sector $\frac{\pi}{2}< \arg z < \frac{3\pi}{2}$, the pole in the first integral in \eqref{eq21} crosses the integration path. According to the residue theorem, we obtain
\begin{align*}
R_{N,M} \left( z \right) & = e^{2\pi iz} R_M \left( z \right) + \frac{1}{2\pi i}\frac{{i^N }}{{z^N }}\int_0^{ + \infty } {\frac{s^{N - 1} e^{ - 2\pi s} R_M \left( {is} \right)}{1 - is/z}ds}  - \frac{1}{2\pi i}\frac{\left( { - i} \right)^N }{z^N}\int_0^{ + \infty } {\frac{s^{N - 1} e^{ - 2\pi s} R_M \left( { - is} \right)}{1 + is/z}ds} \\
& = e^{2\pi iz} R_M \left( z \right) + \widetilde R_{N,M} \left( {ze^{ - \pi i} } \right)
\end{align*}
when $\frac{\pi}{2} < \arg z < \frac{3\pi}{2}$. To find the analytic continuation of $R_M \left( z \right)$ to this sector, we apply the same argument but for the integral representation \eqref{eq15} to deduce
\begin{align*}
R_M \left( z \right) & = e^{2\pi iz} \Gamma^\ast \left( z \right) + \frac{1}{{2\pi i}}\frac{{i^M }}{{z^M }}\int_0^{ + \infty } {\frac{{s^{M - 1} e^{ - 2\pi s} \Gamma^\ast  \left( {is} \right)}}{{1 - is/z}}ds}  - \frac{1}{{2\pi i}}\frac{{\left( { - i} \right)^M }}{{z^M }}\int_0^{ + \infty } {\frac{{s^{M - 1} e^{ - 2\pi s} \Gamma^\ast  \left( { - is} \right)}}{{1 + is/z}}ds} \\
& = e^{2\pi iz} \Gamma^\ast \left( z \right) + \widetilde R_M \left( {ze^{ - \pi i} } \right) = \frac{1}{e^{-2\pi iz}-1 }\frac{1}{{\Gamma^\ast \left( {ze^{ - \pi i} } \right)}} + \widetilde R_M \left( {ze^{ - \pi i} } \right) .
\end{align*}
Here we have made us of the connection formula \eqref{eq28}. Therefore, the analytic continuation of the expansion \eqref{eq49} to the sector $\frac{\pi}{2} < \arg z < \frac{3\pi}{2}$ can be obtained by setting
\[
R_{N,M} \left( z \right) = \frac{{e^{2\pi iz} }}{e^{-2\pi iz} -1}\frac{1}{{\Gamma^\ast \left( {ze^{ - \pi i} } \right)}} + e^{2\pi iz} \widetilde R_M \left( {ze^{ - \pi i} } \right) + \widetilde R_{N,M} \left( {ze^{ - \pi i} } \right).
\]
In the proof of Lemma \ref{lemma2} we showed that the reciprocal scaled gamma function is bounded in the right-half plane, hence by Theorem \ref{thm3} and the estimate \eqref{eq29} we infer that
\begin{align*}
R_{N,M} \left( z \right) & = \mathcal{O}\left( {\frac{{e^{ - 2\pi \Im \left( z \right)} }}{{\left| {1 - e^{ - 2\pi iz} } \right|}}} \right) + \mathcal{O}_M \left( {\frac{{e^{ - 2\pi \Im \left( z \right)} }}{{\left| z \right|^M }}} \right) + \mathcal{O}_{M,\rho } \left( {\frac{{e^{ - 2\pi \left| z \right|} }}{{\left| z \right|^M }}} \right) \\
& = \mathcal{O}_{M,\rho } \left( {e^{ - 2\pi \Im \left( z \right)} \left( {\frac{1}{{\left| {1 - e^{ - 2\pi iz} } \right|}} + \frac{1}{{\left| z \right|^M }}} \right)} \right)
\end{align*}
as $z \to \infty$ in the closed sector $\frac{\pi}{2} \leq \arg z \leq \frac{3\pi}{2}$. The extension to the cases $M=0,1$ follows from \eqref{eq25}, \eqref{eq52} and \eqref{eq53}. Similarly, we find that
\[
\widetilde R_{N,M} \left( z \right) =  - e^{2\pi iz} R_M \left( { ze^{ - \pi i}} \right) + R_{N,M} \left( {ze^{ - \pi i} } \right)
\]
for $\frac{\pi}{2} < \arg z < \frac{3\pi}{2}$,
and therefore
\[
\widetilde R_{N,M} \left( z \right) = \mathcal{O}_M \left( {\frac{{e^{ - 2\pi \Im \left( z \right)} }}{{\left| z \right|^M }}} \right) + \mathcal{O}_{M,\rho } \left( {\frac{{e^{ - 2\pi \left| z \right|} }}{{\left| z \right|^M }}} \right) = \mathcal{O}_{M,\rho } \left( {\frac{{e^{ - 2\pi \Im \left( z \right)} }}{{\left| z \right|^M }}} \right)
\]
as $z \to \infty$ in the sector $\frac{\pi}{2} \leq \arg z \leq \frac{3\pi}{2}$. The proof of the corresponding estimates for the sector $-\frac{3\pi}{2} \leq \arg z \leq - \frac{\pi}{2}$ is completely analogous.

\subsection{Stokes phenomenon and Berry's transition}

It was shown by Paris and Wood \cite[equations (3.2) and (3.4)]{Paris2} that the asymptotic expansions
\begin{equation}\label{eq50}
\log \Gamma^{\ast}  \left( z \right) \sim \sum\limits_{n = 1}^\infty  {\frac{{B_{2n} }}{{2n\left( {2n - 1} \right)z^{2n - 1} }}} - \begin{cases} 0 & \text{ if } \; \left|\arg z\right| <\frac{\pi}{2} \\ \frac{1}{2}\log \left( {1 - e^{ \pm 2\pi iz} } \right) & \text{ if } \; \arg z = \pm\frac{\pi}{2} \\ \log \left( {1 - e^{ \pm 2\pi iz} } \right) & \text{ if } \; \frac{\pi}{2} < \pm \arg z < \pi, \end{cases}  
\end{equation}
hold as $z \to \infty$. Expanding the logarithm into its Taylor series yields
\begin{equation}\label{eq31}
\log \Gamma^{\ast} \left( {z} \right) \sim \sum\limits_{n  = 1}^\infty  {\frac{B_{2n}}{2n \left( {2n  - 1} \right)z^{2n  - 1} }} + \sum\limits_{k = 1}^\infty  {S_k \left( \theta  \right)e^{ \pm 2\pi ikz} },
\end{equation}
as $z \to \infty$ in the sector $\left|\arg z\right| \leq \pi - \delta < \pi$, for any $0 < \delta \leq \pi$. Here
\begin{equation}\label{eq32}
S_k \left( \theta  \right) = \begin{cases} \frac{1}{k} & \text{ if } \frac{\pi}{2} < \left|\theta\right| <\pi \\ \frac{1}{2k} & \text{ if } \theta = \pm \frac{\pi}{2} \\ 0 & \text{ if } \left|\theta\right| < \frac{\pi}{2} \end{cases}
\end{equation}
are the Stokes multipliers and $\theta = \arg z$. The upper or lower sign is taken in \eqref{eq31} and \eqref{eq32} according as $z$ is in the upper or lower half-plane. Taking the exponential of both sides in \eqref{eq31}, we arrive at the expansions
\begin{equation}\label{eq33}
\Gamma^{\ast}\left( z \right) \sim \sum\limits_{n  = 0}^\infty  {\left( { - 1} \right)^n  \frac{\gamma_n}{z^n}} + \mathscr{S}_1 \left( \theta  \right)e^{ \pm 2\pi iz} \sum\limits_{n  = 0}^\infty  {\left( { - 1} \right)^n  \frac{\gamma_n}{z^n}}  +  \cdots + \mathscr{S}_k \left( \theta  \right)e^{ \pm 2\pi ikz} \sum\limits_{n  = 0}^\infty  {\left( { - 1} \right)^n  \frac{\gamma_n}{z^n}}  +  \cdots
\end{equation}
and
\begin{equation}\label{eq34}
\frac{1}{\Gamma^\ast \left( z \right)} \sim \sum\limits_{n = 0}^\infty  {\frac{{\gamma _n }}{{z^n }}}  + \widetilde{ \mathscr{S}}_1 \left( \theta  \right)e^{ \pm 2\pi iz} \sum\limits_{n = 0}^\infty  {\frac{{\gamma _n }}{{z^n }}}  +  \cdots  + \widetilde{\mathscr{S}}_k \left( \theta  \right)e^{ \pm 2\pi i k z} \sum\limits_{n = 0}^\infty  {\frac{{\gamma _n }}{{z^n }}}  +  \cdots ,
\end{equation}
as $z \to \infty$ in the sector $\left|\arg z\right| \leq \pi - \delta <\pi$, $0 < \delta \leq \pi$, with the Stokes multipliers
\[
\mathscr{S}_k \left( \theta  \right) = \begin{cases} 1 & \text{ if } \frac{\pi}{2} < \left|\theta\right| <\pi \\ \frac{1}{k!}\left(\frac{1}{2}\right)_k & \text{ if } \theta = \pm \frac{\pi}{2} \\ 0 & \text{ if } \left|\theta\right| < \frac{\pi}{2} \end{cases}
\]
and
\[
\widetilde{\mathscr{S}}_1 \left( \theta  \right) = \begin{cases} -1 & \text{ if } \frac{\pi}{2} < \left|\theta\right| <\pi \\ - \frac{1}{2} & \text{ if } \theta = \pm \frac{\pi}{2} \\ 0 & \text{ if } \left|\theta\right| < \frac{\pi}{2} \end{cases}, \;
\widetilde{\mathscr{S}}_k \left( \theta  \right) = \begin{cases} 0 & \text{ if } \frac{\pi}{2} < \left|\theta\right| <\pi \\ \frac{1}{k!}\left(-\frac{1}{2}\right)_k & \text{ if } \theta = \pm \frac{\pi}{2} \\ 0 & \text{ if } \left|\theta\right| < \frac{\pi}{2} \end{cases} \; \text{ for } \; k\geq 2.
\]
Here $\left(x\right)_k = \Gamma\left(x+k\right)/\Gamma\left(x\right)$ stands for the Pochhammer symbol \cite[p. 136]{NIST}. It is seen that there is a discontinuous change in the coefficients of the exponential terms when $\arg z$ changes continuously across $\arg z = \pm \frac{\pi}{2}$. We have encountered a Stokes phenomenon with Stokes lines $\arg z = \pm \frac{\pi}{2}$. The formulas for $\mathscr{S}_1 \left( \theta  \right)$ and $\widetilde{\mathscr{S}}_1 \left( \theta  \right)$ are in agreement with Dingle's non-rigorous ``final main rule" in his theory of terminants \cite[p. 414]{Dingle}, namely that half the discontinuity occurs on reaching the Stokes ray, and half on leaving it the other side. However, for the higher-order Stokes multipliers this rule is no longer valid.

The interesting behaviour of the asymptotic series for the reciprocal gamma function is worth noting. On the Stokes lines $\arg z = \pm \frac{\pi}{2}$ infinitely many subdominant exponential terms appear in the expansion and as $\arg z$ passes through the values $\pm\frac{\pi}{2}$, all but one of them disappear.

In the important paper \cite{Berry2}, Berry provided a new interpretation of the Stokes phenomenon; he found that assuming optimal truncation, the transition between compound asymptotic expansions is of Error function type, thus yielding a smooth, although very rapid, transition as a Stokes line is crossed.

Using the exponentially improved expansions given in Theorem \ref{thm4}, we show that the asymptotic expansions of the gamma function and its reciprocal exhibit the Berry transition between the two asymptotic series across the Stokes lines $\arg z = \pm \frac{\pi}{2}$. More precisely, we shall find that the first few terms of the series in \eqref{eq33} and \eqref{eq34} corresponding to the subdominant exponentials $e^{ \pm 2\pi iz}$ ``emerge" in a rapid and smooth way as $\arg z$ passes through $\pm\frac{\pi}{2}$.

From Theorem \ref{thm4}, we conclude that if $N \approx 2\pi\left|z\right|$ then for large $z$, $\left|\arg z\right| < \pi$, we have
\[
\Gamma^\ast \left( z \right) \approx \sum\limits_{n = 0}^{N - 1} {\left( { - 1} \right)^n \frac{{\gamma _n }}{{z^n }}}  + e^{2\pi iz} \sum\limits_{m = 0} {\left( { - 1} \right)^m \frac{{\gamma _m }}{{z^m }}\widehat T_{N - m} \left( {2\pi iz} \right)}  - e^{ - 2\pi iz} \sum\limits_{m = 0} {\left( { - 1} \right)^m \frac{{\gamma _m }}{{z^m }}\widehat T_{N - m} \left( { - 2\pi iz} \right)} 
\]
and
\[
\frac{1}{\Gamma^\ast  \left( z \right)} \approx \sum\limits_{n = 0}^{N - 1} {\frac{{\gamma _n }}{{z^n }}}  - e^{2\pi iz} \sum\limits_{m = 0} {\frac{{\gamma _m }}{{z^m }}\widehat T_{N - m} \left( {2\pi iz} \right)}  + e^{ - 2\pi iz} \sum\limits_{m = 0} {\frac{{\gamma _m }}{{z^m }}\widehat T_{N - m} \left( { - 2\pi iz} \right)} ,
\]
where $\sum\nolimits_{m = 0}$ means that the sum is restricted to the leading terms of the series. In the upper half-plane the terms involving $\widehat T_{N - m} \left( { - 2\pi iz} \right)$ are exponentially small, the dominant contribution comes from the terms involving $\widehat T_{N - m} \left( { 2\pi iz} \right)$. Under the above assumptions on $N$, from \eqref{eq35} and \eqref{eq27}, the Terminant functions have the asymptotic behaviour
\[
\widehat T_{N - m} \left( {2\pi iz} \right) \sim \frac{1}{2} + \frac{1}{2}\mathop{\text{erf}}\left( {\left( {\theta  - \frac{\pi }{2}} \right)\sqrt {\pi \left| z \right|} } \right)
\]
provided that $\arg z = \theta$ is close to $\frac{\pi}{2}$, $z$ is large and $m$ is small in comparison with $N$. Therefore, when $\theta  < \frac{\pi}{2}$, the Terminant functions are exponentially small; for $\theta  = \frac{\pi }{2}$, they are asymptotically $\frac{1}{2}$ up to an exponentially small error; and when $\theta  >  \frac{\pi}{2}$, the Terminant functions are asymptotic to $1$ with an exponentially small error. Thus, the transition across the Stokes line $\arg z = \frac{\pi}{2}$ is effected rapidly and smoothly. Similarly, in the lower half-plane, the dominant contribution is controlled by the terms involving $\widehat T_{N - m} \left( { - 2\pi iz} \right)$. From \eqref{eq26} and \eqref{eq27}, we have
\[
\widehat T_{N - m} \left( { - 2\pi iz} \right) \sim - \frac{1}{2} + \frac{1}{2} \mathop{\text{erf}} \left( {\left( {\theta  + \frac{\pi }{2}} \right)\sqrt {\pi \left| z \right|} } \right)
\]
under the assumptions that $\arg z = \theta$ is close to $-\frac{\pi}{2}$, $z$ is large and $m$ is small in comparison with $N \approx 2\pi\left|z\right|$. Thus, when $\theta  >  - \frac{\pi}{2}$, the Terminant functions are exponentially small; for $\theta  =  -\frac{\pi}{2}$, they are asymptotic to $-\frac{1}{2}$ with an exponentially small error; and when $\theta < - \frac{\pi}{2}$, the Terminant functions are asymptotically $-1$ up to an exponentially small error. Therefore, the transition through the Stokes line $\arg z = -\frac{\pi}{2}$ is carried out rapidly and smoothly.

We remark that the smooth transition of the subdominant exponential $e^{2\pi iz}$ was also discussed by Boyd \cite{Boyd2}, though he used a slightly different approximation for the Terminant functions.

\section{The asymptotics of the late coefficients}\label{section4}

The asymptotic form of the Stirling coefficients $\gamma_n$ is well known. Their leading order behaviour was investigated by Watson \cite{Watson} using the method of Darboux, and by Diekmann \cite{Diekmann} using the method of steepest descents. Murnaghan and Wrench \cite{Murnaghan} gave higher approximations employing Darboux's method. Complete asymptotic expansions were derived by Dingle \cite[pp. 158--159]{Dingle}, though, his results were obtained by formal and interpretive, rather than rigorous, methods. His expansions may be written, in our notation,
\begin{equation}\label{eq45}
\gamma _{2n - 1}  \approx \frac{{\left( { - 1} \right)^n 2}}{{\left( {2\pi } \right)^{2n} }}\sum\limits_{k = 0}^\infty  {\left( { - 1} \right)^k \gamma _{2k} \left( {2\pi } \right)^{2k} \Gamma \left( {2n - 2k - 1} \right)\zeta \left( {2n - 2k} \right)} 
\end{equation}
and
\begin{equation}\label{eq42}
\gamma _{2n}  \approx \frac{{\left( { - 1} \right)^n 2}}{{\left( {2\pi } \right)^{2n} }}\sum\limits_{k = 0}^\infty  {\left( { - 1} \right)^k \gamma _{2k + 1} \left( {2\pi } \right)^{2k} \Gamma \left( {2n - 2k - 1} \right)\zeta \left( {2n - 2k} \right)} .
\end{equation}
Finally, Boyd \cite{Boyd2} gave expansions similar to Dingle's complete with error bounds, using truncated forms of the approximations 
\[
\Gamma^\ast \left( z \right) \sim \sum\limits_{k = 0}^\infty  {\left( { - 1} \right)^k \frac{\gamma_k}{z^k }} \; \text{ and } \; \Gamma^\ast \left( z \right) \sim \frac{1}{1 - e^{2\pi i z} }\sum\limits_{k = 0}^\infty  {\left( { - 1} \right)^k \frac{\gamma_k}{z^k}}
\]
with $z=is$ ($s>0$), and his resurgence formula for the gamma function. Although both expansions are valid along the positive imaginary axis in Poincar\'{e}'s sense, from \eqref{eq50} it is seen that the first one is more suitable when $\left|\arg z\right| < \frac{\pi}{2}$ and the second one is more suitable when $\frac{\pi}{2} < \arg z <\pi$. In our notations, Boyd's results can be written as
\begin{equation}\label{eq46}
\gamma _{2n - 1} = \frac{{\left( { - 1} \right)^n 2}}{{\left( {2\pi } \right)^{2n} }}\sum\limits_{k = 0}^{K-1}  {\left( { - 1} \right)^k \gamma _{2k} \left( {2\pi } \right)^{2k} \Gamma \left( {2n - 2k - 1} \right)} + M_{K}\left(2n-1\right),
\end{equation}
\begin{equation}\label{eq43}
\gamma _{2n} = \frac{{\left( { - 1} \right)^n 2}}{{\left( {2\pi } \right)^{2n} }}\sum\limits_{k = 0}^{K-1}  {\left( { - 1} \right)^k \gamma _{2k + 1} \left( {2\pi } \right)^{2k} \Gamma \left( {2n - 2k - 1} \right)} + M_{K}\left(2n\right)
\end{equation}
and
\begin{equation}\label{eq47}
\gamma _{2n - 1} = \frac{{\left( { - 1} \right)^n 2}}{{\left( {2\pi } \right)^{2n} }}\sum\limits_{k = 0}^{K-1}  {\left( { - 1} \right)^k \gamma _{2k} \left( {2\pi } \right)^{2k} \Gamma \left( {2n - 2k - 1} \right)\zeta \left( {2n - 2k - 1} \right)} + \widehat{M}_{K}\left(2n-1\right),
\end{equation}
\begin{equation}\label{eq44}
\gamma _{2n} = \frac{{\left( { - 1} \right)^n 2}}{{\left( {2\pi } \right)^{2n} }}\sum\limits_{k = 0}^{K-1}  {\left( { - 1} \right)^k \gamma _{2k + 1} \left( {2\pi } \right)^{2k} \Gamma \left( {2n - 2k - 1} \right)\zeta \left( {2n - 2k - 1} \right)} + \widehat{M}_{K}\left(2n\right).
\end{equation}
Here $1 \le K < n$ and the truncation errors $M_K$ and $\widehat{M}_{K}$ can be bounded explicitly and realistically.

Boyd observed that estimating the error term $M_{K}\left(2n-1\right)$ in \eqref{eq46} via the exponentially improved expansion of the gamma function \eqref{eq48} along the imaginary axis, leads to an improved version of the late coefficient formula \eqref{eq46}. His improved expansion \cite[equation (3.42)]{Boyd2} also shed some light on the idea behind Dingle's formula \eqref{eq45}, especially on the appearance of the zeta function of Riemann in the approximation and its numerical superiority over Boyd's formula \eqref{eq46}.

The main goal of this section is to derive new asymptotic series for the Stirling coefficients using the representations \eqref{eq7} and \eqref{eq8} and an exponentially improved asymptotic expansion for the gamma function. These new expansions utilise all the exponentially small terms in \eqref{eq33} and provide a full explanation of the remarkable accuracy of Dingle's series \eqref{eq45} and \eqref{eq42}. From \eqref{eq50}, we see that
\begin{equation}\label{eq16}
\Gamma^\ast \left( {is} \right) \sim \frac{1}{\sqrt {1 - e^{ - 2\pi s} }}\sum\limits_{k = 0}^\infty  {\left( { - 1} \right)^k \frac{\gamma_k}{\left( {is} \right)^k }} 
\end{equation}
as $s \to +\infty$. Consequently, we have
\[
\Re \Gamma^\ast \left( {is} \right) \sim \frac{1}{\sqrt {1 - e^{ - 2\pi s} }}\sum\limits_{k = 0}^\infty  {\left( { - 1} \right)^k \frac{\gamma _{2k}}{s^{2k}}} 
\]
and
\[
\Im \Gamma^\ast \left( {is} \right) \sim \frac{1}{\sqrt {1 - e^{ - 2\pi s} }}\sum\limits_{k = 0}^\infty  {\left( { - 1} \right)^k \frac{\gamma _{2k + 1}}{s^{2k + 1}}} 
\]
as $s \to +\infty$. Substituting these expansions into \eqref{eq7} and \eqref{eq8} yields the formal asymptotic series
\begin{equation}\label{eq17}
\gamma _{2n - 1}  \approx \frac{{\left( { - 1} \right)^n 2}}{{\left( {2\pi } \right)^{2n} }}\sum\limits_{k = 0}^\infty  {\left( { - 1} \right)^k \gamma _{2k} \left( {2\pi } \right)^{2k} \Gamma \left( {2n - 2k - 1} \right)\xi \left( {2n - 2k - 1} \right)} 
\end{equation}
and
\begin{equation}\label{eq18}
\gamma _{2n}  \approx \frac{{\left( { - 1} \right)^n 2}}{{\left( {2\pi } \right)^{2n} }}\sum\limits_{k = 0}^\infty  {\left( { - 1} \right)^k \gamma _{2k + 1} \left( {2\pi } \right)^{2k} \Gamma \left( {2n - 2k - 1} \right)\xi \left( {2n - 2k - 1} \right)} 
\end{equation}
for large $n$. Here, the function $\xi \left( r \right)$ is given by the Dirichlet series
\begin{align*}
\xi \left( r \right) = \frac{{\left( {2\pi } \right)^r }}{{\Gamma \left( r \right)}}\int_0^{ + \infty } {\frac{{s^{r - 1} e^{ - 2\pi s} }}{{\sqrt {1 - e^{ - 2\pi s} } }}ds} & = \sum\limits_{m = 0}^\infty  {\frac{\left(\frac{1}{2}\right)_m}{m! \left( {m + 1} \right)^r}} 
\\ & = 1 + \frac{1}{2}\frac{1}{{2^r }} + \frac{3}{8}\frac{1}{{3^r }} + \frac{5}{{16}}\frac{1}{{4^r }} + \frac{{35}}{{128}}\frac{1}{{5^r }} +  \cdots ,
\end{align*}
provided that $r > \frac{1}{2}$. The formal expansions in \eqref{eq17} and \eqref{eq18} can be turned into exact results by constructing error bounds for the series \eqref{eq16}, but we do not pursue the details here. We shall assume that optimal truncation of these series provides good approximations for the Stirling coefficients $\gamma_n$.

For large $n$ and fixed $k$, we have
\[
\zeta \left( 2n - 2k \right) \approx 1 + \frac{1}{2^{2n - 2k}} + \frac{1}{3^{2n - 2k}}
\]
\[
\zeta \left( 2n - 2k - 1 \right) \approx 1 + 2\frac{1}{2^{2n - 2k}} + 3\frac{1}{3^{2n - 2k}},
\]
and
\[
\xi \left( 2n - 2k - 1 \right) \approx 1 + \frac{1}{2^{2n - 2k}} + \frac{9}{8}\frac{1}{3^{2n - 2k}}.
\]
These approximations explain Boyd's observation, namely that, assuming optimal truncation, Dingle's expansions provide better approximations than Boyd's original series. We also get a numerical explanation of the appearance of Riemann's zeta function in Dingle's expansions.

We remark that Boyd's improved series \cite[equation (3.42)]{Boyd2} for $\gamma_{2n-1}$ is \eqref{eq17} with the approximate values
\[
\xi \left( {2n - 2k - 1} \right) \approx \begin{cases} 1 + 2^{-2n} & \text{ if } \; k = 0 \\ 1 & \text{ if } \; k>0. \end{cases}
\]

In our calculations we truncated the expansions of Dingle and ours at $k=K-1$ like Boyd did and chose the value of $K$ optimally. Optimality here means that we choose $K$ in terms of $n$ such that the last term of the remaining series is asymptotically the smallest in absolute value for large $n$. It can be shown that the optimal choice of $K$ for both the expansions of $\gamma_{2n-1}$ and $\gamma_{2n}$ is $K=\left\lceil \frac{n}{2}\right\rceil$, i.e., in the sums $k$ runs from $0$ to $\left\lceil \frac{n}{2}\right\rceil-1$. Tables \ref{table1} and \ref{table2} display the numerical results obtained for the coefficients $\gamma_{101}$ and $\gamma_{100}$ by using the optimally truncated approximations of Dingle, Boyd and ours.

\begin{table*}[!ht]
\begin{center}
\begin{tabular}
[c]{ l r @{\,}c@{\,} l}\hline
 & \\ [-1ex]
 values of $n$ and $K$ & $n=51$, $K=26$ & &  \\ [1ex]
 exact numerical value of $\gamma_{2n-1}$ & $-0.718920823005286472090671337669485196245$ & $\times$ & $10^{77}$ \\ [1ex]
 Dingle's approximation \eqref{eq45} to $\gamma_{2n-1}$ & $-0.718920823005286472090671337669485196372$ & $\times$ & $10^{77}$ \\ [1ex]
 error & $0.127$ & $\times$ & $10^{41}$\\ [1ex]
 Boyd's approximation \eqref{eq46} to $\gamma_{2n-1}$ & $-0.718920823005286472090671337669343420137$ & $\times$ & $10^{77}$ \\ [1ex]
 error & $-0.141776108$ & $\times$ & $10^{47}$\\ [1ex]
 Boyd's approximation \eqref{eq47} to $\gamma_{2n-1}$ & $-0.718920823005286472090671337669626972607$ & $\times$ & $10^{77}$ \\ [1ex]
 error & $0.141776362$ & $\times$ & $10^{47}$\\ [1ex]
 approximation \eqref{eq17} to $\gamma_{2n-1}$ & $-0.718920823005286472090671337669485196372$ & $\times$ & $10^{77}$ \\ [1ex]
 error & $0.127$ & $\times$ & $10^{41}$\\ [-1ex]
 & \\\hline
\end{tabular}
\end{center}
\caption{Approximations for $\gamma_{101}$, using optimal truncation.}
\label{table1}
\end{table*}

\begin{table*}[!ht]
\begin{center}
\begin{tabular}
[c]{ l r @{\,}c@{\,} l}\hline
 & \\ [-1ex]
 values of $n$ and $K$ & $n=50$, $K=25$ & &  \\ [1ex]
 exact numerical value of $\gamma_{2n}$ & $-0.238939789661593595677447537129753012$ & $\times$ & $10^{74}$ \\ [1ex]
 Dingle's approximation \eqref{eq42} to $\gamma_{2n}$ & $-0.238939789661593595677447537129753175$ & $\times$ & $10^{74}$ \\ [1ex]
 error & $0.163$ & $\times$ & $10^{41}$\\ [1ex]
 Boyd's approximation \eqref{eq43} to $\gamma_{2n}$ & $-0.238939789661593595677447537129564608$ & $\times$ & $10^{74}$ \\ [1ex]
 error & $-0.188403$ & $\times$ & $10^{44}$\\ [1ex]
 Boyd's approximation \eqref{eq44} to $\gamma_{2n}$ & $-0.238939789661593595677447537129941741$ & $\times$ & $10^{74}$ \\ [1ex]
 error & $0.188729$ & $\times$ & $10^{44}$\\ [1ex]
 approximation \eqref{eq18} to $\gamma_{2n}$ & $-0.238939789661593595677447537129753175$ & $\times$ & $10^{74}$ \\ [1ex]
 error & $0.163$ & $\times$ & $10^{41}$\\ [-1ex]
 & \\\hline
\end{tabular}
\end{center}
\caption{Approximations for $\gamma_{100}$, using optimal truncation.}
\label{table2}
\end{table*}

It is seen from the numerical computations that our expansions provide better approximations than that of Boyd's and are comparable with the expansions of Dingle. We remark that the approximate numerical value of $\gamma_{100}$ arising from Boyd's formula \eqref{eq44} was given incorrectly in Table 4 of his paper \cite{Boyd2}.

\appendix

\section{The computation of the Stirling coefficients}\label{appendix}

In this appendix, we collect some known recurrence representations of the Stirling coefficients $\gamma_n$. The exact values of $\gamma_n$ up to $\gamma_{30}$ can be found in the papers of Spira \cite{Spira} and Wrench \cite{Wrench}. Explicit formulas for the Stirling coefficients are given by Boyd \cite{Boyd4}, Brassesco and M\'endez \cite{Brassesco}, Comtet \cite[p. 267]{Comtet}, De Angelis \cite{De Angelis}, L\'{o}pez, Pagola and P\'{e}rez Sinus\'\i a \cite{Lopez}, Nemes \cite{Nemes} and Wrench \cite{Wrench}. 

\subsection{Recurrence relations}
Based on the Lagrange inversion formula, Brassesco and M\'endez \cite{Brassesco} find recursive formulas for the computation of the Stirling coefficients. Define the sequence $b_n$ by the recurrence relation
\begin{equation}\label{eq36}
b_n  = \frac{{2 - n}}{{3n + 3}}b_{n - 1}  - \frac{1}{{n + 1}}\sum\limits_{k = 2}^{n - 3} {\left( {k + 1} \right)b_{k + 1} b_{n - k} } ,
\end{equation}
for $n \ge 4$ with $b_1  = 1$, $b_3 = \frac{1}{36}$. Then the coefficients $\gamma_n$ can be computed as
\[
\gamma _n  = \left( { - 1} \right)^n \frac{\left( {2n + 1} \right)!}{2^n n!}b_{2n + 1} .
\]
This recurrence was also given by Borwein and Corless \cite{Borwein}.

Upon replacing $k$ by $n-k-1$ in the sum, we see that the recurrence relation \eqref{eq36} may be written in the form
\[
b_n  = \frac{{2 - n}}{{3n + 3}}b_{n - 1}  - \frac{1}{2}\sum\limits_{k = 2}^{n - 3} {b_{k + 1} b_{n - k} } .
\]
This formula was also found by Copson \cite[p. 56]{Copson}.

Wrench \cite{Wrench} found recurrence formulas in terms of the Bernoulli numbers $B_k$ since
\[
\gamma _{2n - 1}  =  - \frac{1}{2n - 1}\sum\limits_{k = 1}^n {\frac{B_{2k}}{2k}\gamma _{2n - 2k} } \; \text{ and } \; 
\gamma _{2n}  =  - \frac{1}{2n}\sum\limits_{k = 1}^n {\frac{B_{2k}}{2k}\gamma _{2n - 2k + 1} } ,
\]
for $n \ge 1$ with $\gamma_0 = 1$. To derive these results, he used the formal generating function
\begin{equation}\label{eq37}
\exp \left( {\sum\limits_{n = 1}^\infty  {\frac{B_{2n}}{2n\left( {2n - 1} \right)}x^{2n - 1} } } \right) = \exp \left( {\sum\limits_{n = 1}^\infty  {\frac{B_{n + 1}}{n\left( {n + 1} \right)}x^n } } \right) = \sum\limits_{n = 0}^\infty  {\left( { - 1} \right)^n \gamma _n x^n }
\end{equation}
which follows from \eqref{eq4}. We derive here another type of recurrence formulas using the generating function \eqref{eq37}. Differentiating both sides of \eqref{eq37} with respect to $x$ and dividing each side by the exponential expression on the left-hand side of \eqref{eq37}, we find
\[
\sum\limits_{n = 1}^\infty  {\frac{{B_{2n} }}{{2n}}x^{2n - 2} }  = \exp \left( { - \sum\limits_{n = 1}^\infty  {\frac{{B_{2n} }}{{2n\left( {2n - 1} \right)}}x^{2n - 1} } } \right)\sum\limits_{n = 1}^\infty  {\left( { - 1} \right)^n n\gamma _n x^{n - 1} } .
\]
Noting that
\begin{equation}\label{eq38}
\exp \left( { - \sum\limits_{n = 1}^\infty  {\frac{{B_{2n} }}{{2n\left( {2n - 1} \right)}}x^{2n - 1} } } \right) = \exp \left( {\sum\limits_{n = 1}^\infty  {\frac{{B_{2n} }}{{2n\left( {2n - 1} \right)}}\left( { - x} \right)^{2n - 1} } } \right) = \sum\limits_{n = 0}^\infty  {\gamma _n x^n } ,
\end{equation}
and equating the coefficients of equal powers of $x$ we deduce the recursive formulas
\begin{equation}\label{eq39}
\gamma _{2n - 1}  =  - \frac{{B_{2n} }}{{2n\left( {2n - 1} \right)}} + \frac{1}{{2n - 1}}\sum\limits_{k = 1}^{2n - 2} {\left( { - 1} \right)^k k\gamma _k \gamma _{2n - k - 1} } \; \text{ and } \;
\gamma _{2n}  =  - \frac{1}{{2n}}\sum\limits_{k = 1}^{2n - 1} {\left( { - 1} \right)^k k\gamma _k \gamma _{2n - k} } ,
\end{equation}
for $n \ge 1$ with $\gamma_0 = 1$. From \eqref{eq37} and \eqref{eq38} we can immediately obtain the known expression
\[
\sum\limits_{k = 0}^n {\left( { - 1} \right)^k \gamma _k \gamma _{n - k} }  = 0
\]
for $n \ge 1$ (see, e.g., Paris and Kaminski \cite[p. 33]{Paris3}). When $n$ is odd, this is a simple identity, for $n \geq 2$ even it gives
\[
\gamma _{n}  =  - \frac{1}{2}\sum\limits_{k = 1}^{n - 1} {\left( { - 1} \right)^k \gamma _k \gamma _{n - k} } 
\]
which is equivalent to the second recurrence in \eqref{eq39}.

\subsection{Representations using polynomial sequences}
In 1952, Lauwerier \cite{Lauwerier} showed that the coefficients in asymptotic expansions of Laplace-type integrals can be calculated by means of linear recurrence relations. As an illustration of his method, he considered, inter alia, the Stirling coefficients $\gamma_n$. Define the sequence of polynomials $P_0 \left( x \right), P_1 \left( x \right), P_2 \left( x \right),\ldots$ via the recurrence
\[
P_n \left( x \right) =  - P_{n - 1} \left( x \right) + \frac{{x^{ - \frac{n}{2}} }}{2}\int_0^x {t^{\frac{n}{2}} P_{n - 1} \left( t \right)dt} 
\]
for $n \ge 1$ with $P_0 \left( x \right) = 1$. Then the coefficients $\gamma_n$ can be recovered from the formula
\[
\gamma _n  = \frac{{\left( { - 1} \right)^n }}{{\sqrt {2\pi } }}\int_0^{ + \infty } {e^{ - \frac{t}{2}} t^{n - \frac{1}{2}} P_{2n} \left( t \right)dt} .
\]
It is known that the Stirling coefficients are related to certain polynomials $U_n \left( x \right)$ appearing in the uniform asymptotic expansions of the modified Bessel functions. These polynomials satisfy the following recurrence
\[
U_n \left( x \right) = \frac{1}{2}x^2 \left( {1 - x^2 } \right)U'_{n-1} \left( x \right) + \frac{1}{8}\int_0^x {\left( {1 - 5t^2 } \right)U_{n-1} \left( t \right)dt} 
\]
for $n \ge 1$ with $U_0\left(x\right) = 1$. The coefficients $\gamma_n$ are then given by $\gamma_n = U_n\left(1\right)$ (see, e.g., Olver \cite{Olver3}).

\section*{Acknowledgement} I would like to thank the anonymous referee for his/her thorough, constructive and helpful comments and suggestions on the manuscript.

\end{document}